\documentclass[11pt,reqno]{amsart}

\usepackage[letterpaper, margin=1in]{geometry}
\usepackage{amsfonts}
\usepackage{amssymb}
\usepackage{amsmath}
\usepackage{mathrsfs}
\usepackage{color}
\usepackage[noadjust]{cite}
\usepackage{verbatim}
\usepackage{dsfont}
\usepackage{bbm}
\usepackage{enumerate}
\usepackage{graphicx}
\usepackage{xcolor} 
\usepackage{slashed}
\usepackage[titletoc,title]{appendix}
\usepackage{wrapfig}
\usepackage{marginnote}
\usepackage{bm}
\usepackage{enumitem}

\definecolor{deepgreen}{cmyk}{1,0,1,0.5}

\newcommand{\Del}[1]{}

\numberwithin{equation}{section}

\newtheorem{theorem}{Theorem}[section]
\newtheorem{lemma}[theorem]{Lemma}
\newtheorem{proposition}[theorem]{Proposition}
\newtheorem{remark}[theorem]{Remark}

\renewcommand{\hbar}{{\underline h}}

\newcommand{\bbN}{\mathbb N}

\newcommand{\bbR}{\mathbb R}

\newcommand{\calA}{\mathcal A}
\newcommand{\calB}{\mathcal B}
\newcommand{\calC}{\mathcal C}
\newcommand{\calD}{\mathcal D}

\newcommand{\calF}{\mathcal F}

\newcommand{\calH}{\mathcal H}
\newcommand{\calI}{\mathcal I}
\newcommand{\calJ}{\mathcal J}
\newcommand{\calK}{\mathcal K}

\newcommand{\calM}{\mathcal M}

\newcommand{\calR}{\mathcal R}

\newcommand{\calT}{\mathcal T}

\newcommand{\calY}{\mathcal Y}
\newcommand{\calZ}{\mathcal Z}

\newcommand{\tilu}{{\tilde{u}}}

\newcommand{\tilv}{{\tilde{v}}}

\newcommand{\wtilK}{{\widetilde{K}}}

\newcommand{\hatf}{{\hat{f}}}

\newcommand{\ud}{\mathrm{d}}

\newcommand{\px}{\partial_x}

\newcommand{\pt}{\partial_t}

\newcommand{\hf}{\frac{1}{2}}
\newcommand{\thf}{\frac{3}{2}}

\DeclareMathOperator{\sech}{sech}
\DeclareMathOperator{\cosech}{cosech}

\newcommand{\R}{{\mathbb R}}

\newcommand{\N}{{\mathbb N}}

\newcommand{\eps}{{\varepsilon}}

\begin{document}


\title[Soliton dynamics for the 1D quadratic KG equation with symmetry]{Soliton dynamics for the 1D quadratic Klein-Gordon equation with symmetry}

\author{Yongming Li}
\address{Department of Mathematics \\ Texas A\&M University \\ College Station, TX 77843, USA}
\email{liyo0008@tamu.edu}

\author{Jonas L\"uhrmann}
\address{Department of Mathematics \\ Texas A\&M University \\ College Station, TX 77843, USA}
\email{luhrmann@math.tamu.edu}

\thanks{
J. L\"uhrmann was partially supported by NSF grant DMS-1954707. 
}

\begin{abstract}
We establish the conditional asymptotic stability in a local energy norm of the unstable soliton for the one-dimensional quadratic Klein-Gordon equation under even perturbations.
A key feature of the problem is the positive gap eigenvalue exhibited by the linearized operator around the soliton.
Our proof is based on several virial-type estimates, combining techniques from the series of works \cite{KMM17, KMM17_short, KMM19, KMMV20, KM22}, and an explicitly verified Fermi Golden Rule.
The approach hinges on the fact that even perturbations are orthogonal to the odd threshold resonance of the linearized operator.
\end{abstract}

\maketitle 

\tableofcontents

\section{Introduction}

\subsection{Main result}

We consider the one-dimensional quadratic Klein-Gordon equation
\begin{equation} \label{equ:quadratic_KG}
 (\pt^2 - \px^2 +1) \phi = \phi^2, \quad (t,x) \in \bbR \times \bbR.
\end{equation}
The model enjoys space-time translation invariance and invariance under Lorentz transformations.
Its solutions formally conserve the energy
\begin{equation*}
 E = \int_\bbR \biggl( \frac12 (\px \phi)^2 + \frac12 (\pt \phi)^2 + \frac12 \phi^2 - \frac13 \phi^3 \biggr) \, \ud x.
\end{equation*}
Introducing $\bm{\phi} = (\phi, \pt \phi) = (\phi_1, \phi_2)$, we can also write~\eqref{equ:quadratic_KG} as a first-order system 
\begin{equation} \label{equ:quadratic_KG_system}
 \left\{ \begin{aligned}
  \pt \phi_1 &= \phi_2, \\
  \pt \phi_2 &= -(-\px^2 + 1) \phi_1 + \phi_1^2.
 \end{aligned} \right.
\end{equation}
Local well-posedness for arbitrary finite energy data follows from a standard fixed-point argument and the global existence of solutions for initial data with small energy can be deduced using the conservation of energy.
Observe that for even initial data the parity is preserved under the flow of~\eqref{equ:quadratic_KG_system}.
In this work we consider the dynamics of even solutions $\bm{\phi}$ to~\eqref{equ:quadratic_KG_system} in the vicinity of the even soliton solution $\bm{Q} = (Q,0)$ with
\begin{equation*}
 Q(x) = \frac{3}{2} \sech^2\Bigl(\frac{x}{2}\Bigr).
\end{equation*}
More precisely, we study the (conditional) asymptotic stability properties of the soliton $\bm{Q}$ under small even perturbations.
Correspondingly, we decompose
\begin{equation*}
 \bm{\phi} = \bm{Q} + \bm{\varphi}
\end{equation*}
and write~\eqref{equ:quadratic_KG_system} in terms of $\bm{\varphi}=(\varphi_1, \varphi_2)$ as
\begin{equation} \label{equ:linearized_KG_system}
 \left\{ \begin{aligned}
  \pt \varphi_1 &= \varphi_2, \\
  \pt \varphi_2 &= -L \varphi_1 + \varphi_1^2.
 \end{aligned} \right.
\end{equation}
The linearized operator $L$ around the soliton is given by
\begin{equation} \label{equ:linearized_operator}
 L = -\px^2 - 2Q + 1 = - \partial_x^2 - 3 \sech^2\Bigl(\frac{x}{2}\Bigr) + 1.
\end{equation}
It features the classical Schr\"odinger operator $H = - \partial_x^2 - 3 \sech^2(\frac{x}{2})$, which is a rescaled version of the third member in the hierarchy of reflectionless Schr\"odinger operators with P\"oschl-Teller potentials~\cite{PoschlTeller} given by $H_\ell := -\px^2 - \ell (\ell+1) \sech^2(x)$, $\ell \in \bbN$. Their discrete spectra can be determined explicitly, see for instance \cite[Chapter 4.19]{Titchmarsh_Part1}.
The operator $L$ has essential spectrum $[1,\infty)$ and exhibits an even ($L^2$-normalized) eigenfunction $Y_0$ with negative eigenvalue 
\begin{align*}
Y_0(x) &= c_0 \sech^3\Bigl(\frac{x}{2}\Bigr) , \qquad L Y_0 = -\nu^2 Y_0, \qquad c_0 := \sqrt{\frac{15}{32}}, \qquad \nu^2 = \frac54,\\
\intertext{ an odd ($L^2$-normalized) eigenfunction $Y_1$ with zero eigenvalue,}
Y_1(x) &= c_1 \sech^2\Bigl(\frac{x}{2}\Bigr) \tanh\Bigl(\frac{x}{2}\Bigr) , \qquad L Y_1 = 0, \qquad c_1 := \sqrt{\frac{15}{8}}, \\
\intertext{an even ($L^2$-normalized) eigenfunction $Y_2$ with a positive gap eigenvalue,}
Y_2(x) &= c_2 \sech^3\Bigl(\frac{x}{2}\Bigr) \biggl( 1 - 4 \sinh^2\Bigl(\frac{x}{2}\Bigr) \biggr), \qquad L Y_2 = \mu^2 Y_2, \qquad c_2 := \sqrt{\frac{3}{32}}, \qquad \mu^2 = \frac34, \\
\intertext{and an odd threshold resonance}
 Y_3(x) &= \tanh\Bigl(\frac{x}{2}\Bigr) - \frac52 \sech^2\Bigl(\frac{x}{2}\Bigr) \tanh\Bigl(\frac{x}{2}\Bigr), \qquad L Y_3 = Y_3.
\end{align*}
The eigenfunction $Y_1 \simeq Q'$ is related to the invariance under spatial translations and is referred to as the translational mode.
Since we only consider even perturbations $\bm{\varphi}$ of the soliton, the odd translational mode $Y_1$ is not relevant for our analysis.

The negative eigenvalue of the linearized operator introduces an exponentially unstable mode for the dynamics in the neighborhood of the soliton. Indeed, the linear subsystem of~\eqref{equ:linearized_KG_system} given by
\begin{equation} \label{equ:linearized_KG_subsystem}
 \left\{ \begin{aligned}
  \pt \varphi_1 &= \varphi_2, \\
  \pt \varphi_2 &= - L \varphi_1,
 \end{aligned} \right.
\end{equation}
has the exponentially growing solution 
\begin{equation*}
 \bigl( e^{\nu t} Y_0, \nu e^{\nu t} Y_0 \bigr).
\end{equation*}
One may therefore only hope to establish a conditional asymptotic stability result for the soliton~$\bm{Q}$.
The positive gap eigenvalue of the linearized operator, also called an internal mode, presents a significant obstacle for this.
In fact, the linear subsystem~\eqref{equ:linearized_KG_subsystem} admits the time-periodic solution 
\begin{equation*}
 \bigl( \sin(\mu t) Y_2, \mu \cos(\mu t) Y_2 \bigr),
\end{equation*}
which does not decay. However, it may still be possible for (certain) solutions to the nonlinear system~\eqref{equ:linearized_KG_system} to decay via a nonlinear resonant damping mechanism. The latter should depend on a (nonlinear) resonance condition, which is often referred to as a Fermi Golden Rule.

\medskip

Our main result establishes for even perturbations of the soliton~$\bm{Q}$ that stability in the energy space $H^1(\bbR) \times L^2(\bbR)$ implies asymptotic stability in a spatially localized energy norm.
In what follows, a global solution to~\eqref{equ:quadratic_KG_system} is understood to be a function $\bm{\phi} \in \calC([0,\infty); H^1 \times L^2)$ that satisfies~\eqref{equ:quadratic_KG_system} for all $t \geq 0$.

\begin{theorem} \label{thm:codim_asymptotic_stability}
 There exists $0 < \delta \ll 1$ such that if a global even solution $\bm{\phi}$ to~\eqref{equ:quadratic_KG_system} satisfies
 \begin{equation} \label{equ:thm_assumption_closeness}
  \bigl\| \bm{\phi}(t) - \bm{Q} \bigr\|_{H^1(\bbR) \times L^2(\bbR)} \leq \delta \quad \text{for all } \, t \geq 0, 
 \end{equation}
 then for any bounded interval $I \subset \bbR$, we have 
 \begin{equation*}
  \lim_{t\to\infty} \, \bigl\| \bm{\phi}(t) - \bm{Q} \bigr\|_{H^1(I) \times L^2(I)} = 0.
 \end{equation*}
\end{theorem}

The proof of Theorem~\ref{thm:codim_asymptotic_stability} is based on several virial-type estimates, combining techniques from the remarkable series of works by Kowalczyk-Martel-Mu\~{n}oz~\cite{KMM17, KMM17_short, KMM19}, Kowalczyk-Martel-Mu\~{n}oz-Van den Bosch~\cite{KMMV20}, and Kowalczyk-Martel~\cite{KM22}.

\medskip 

From \cite[Theorem 2]{KMM19} we also obtain a description of the set of initial data leading to global even solutions to~\eqref{equ:quadratic_KG_system} satisfying the stability assumption~\eqref{equ:thm_assumption_closeness}. 
\begin{theorem}[{\protect \cite[Theorem 2]{KMM19}}] \label{thm:orbital_stability} 
There exist $C,\delta_0>0$, a set $\mathcal{A}_0 \subset H^1(\bbR) \times L^2(\bbR)$ given by
\begin{equation*}
 \mathcal{A}_0 := \bigl\{ \bm{\eps} \in H^1(\bbR) \times L^2(\bbR) \, \big| \, \|\bm{\eps} \|_{H^1(\R) \times L^2(\R)} \leq \delta_0, \quad  \text{$\bm{\eps}$ even},\quad \langle \bm{\eps}, \bm{Z_+} \rangle = 0 \bigr\}, \quad \bm{Z_+} := \begin{bmatrix} Y_0 \\ \nu^{-1} Y_0 \end{bmatrix},
\end{equation*}
and a Lipschitz function $h \colon \mathcal{A}_0 \to \R$ with $h(0) = 0$ and $\vert h (\bm{\eps}) \vert \leq C \Vert \bm{\eps} \Vert_{H^1(\R) \times L^2(\R)}^{3/2}$ such that the set
\begin{equation*}
	\mathcal{M} := \bigl\{ \bm{Q} + \bm{\eps} + h(\bm{\eps}) \bm{Y_+} \, \big| \, \bm{\eps} \in \mathcal{A}_0 \bigr\}, \quad \bm{Y_+} := \begin{bmatrix} Y_0 \\ \nu Y_0 \end{bmatrix},
\end{equation*}
has the following properties:
\begin{enumerate}
	\item[(i)] If $\bm{\phi_0} \in \mathcal{M}$, then the solution $\bm{\phi}$ to~\eqref{equ:quadratic_KG_system} with initial data $\bm{\phi}(0) = \bm{\phi_0}$ is global and satisfies
	\begin{equation*} 
 	 \bigl\| \bm{\phi}(t) - \bm{Q} \bigr\|_{H^1(\R) \times L^2(\R)} \leq C \bigl\| \bm{\phi_0} - \bm{Q} \bigr\|_{H^1(\R) \times L^2(\R)} \quad \text{for all } \, t \geq 0.
	\end{equation*}
	\item[(ii)] If a global even solution $\bm{\phi}$ to \eqref{equ:quadratic_KG_system} satisfies
	\begin{equation*} 
		\bigl\| \bm{\phi}(t) - \bm{Q} \bigr\|_{H^1(\R) \times L^2(\R)} \leq \frac{1}{10}\delta_0 \quad \text{for all } t \geq 0,
	\end{equation*}
	then
	\begin{equation}\label{eqn: theorem2.2}
		\bm{\phi}(t) \in \mathcal{M} \quad \text{for all } \, t \geq 0.
	\end{equation}
\end{enumerate}
\end{theorem}
We point out that the statement of \cite[Theorem 2]{KMM19} only pertains to the one-dimensional focusing Klein-Gordon equation~\eqref{equ:focusing_KG} for nonlinearities $p > 3$, but the proof easily adapts to the setting of the one-dimensional quadratic Klein-Gordon equation.
In particular, we emphasize that while the presence of the internal mode is not relevant for obtaining the stability property~\eqref{equ:thm_assumption_closeness} of (certain) solutions, it is a significant difficulty for proving their asymptotic stability.

\medskip 

The investigation of the conditional asymptotic stability of the soliton for the one-dimensional quadratic Klein-Gordon equation should be regarded in the context of the broader study of the soliton dynamics for the family of one-dimensional Klein-Gordon equations
\begin{equation} \label{equ:focusing_KG}
 (\pt^2 - \px^2 + 1) \phi = |\phi|^{p-1} \phi, \quad (t,x) \in \bbR \times \bbR, \quad p > 1,
\end{equation}
which admit the soliton solutions
\begin{equation*}
 Q_p(x) = \bigl( {\textstyle \frac{p+1}{2}} \bigr)^{\frac{1}{p-1}} \sech^{\frac{2}{p-1}}\bigl( {\textstyle \frac{p-1}{2}} x \bigr), \quad p > 1.
\end{equation*}
The corresponding linearized operators are given by
\begin{equation*}
 L_p = -\px^2 - {\textstyle \frac{p(p+1)}{2}} \sech^2\bigl({\textstyle \frac{p-1}{2}} x\bigr) + 1, \quad p > 1.
\end{equation*}
For $p > 3$, the operators $L_p$ exhibit one negative eigenvalue and a zero eigenvalue. In the cubic case $p=3$, an additional threshold resonance emerges. For $1 < p < 3$, more and more positive gap eigenvalues and sometimes threshold resonances emerge as $p \to 1^+$. We refer to~\cite[Section 3]{CGNT07} for a detailed description of the spectrum of the operators $L_p$. See also~\cite[Section 1.3]{KMM19}.
The conditional asymptotic stability of the solitons $Q_p$ under even perturbations was studied by Bizo\'{n}-Chmaj-Szpak~ \cite{BizTadSzp11} via formal and numerical methods. For $p > 5$, Krieger-Nakanishi-Schlag~\cite{KNS12} obtained a global characterization of the dynamics of even solutions to~\eqref{equ:focusing_KG} for energies slightly above the energy of $Q_p$, which in particular includes a conditional asymptotic stability result. See also the monograph~\cite{NakSchlag_Book}. For $p > 3$, Kowalczyk-Martel-Mu\~{n}oz~\cite{KMM19} established the conditional asymptotic stability under even perturbations in a spatially localized energy norm. 

A closely related topic is the study of the asymptotic stability of kink solutions that arise in $(1+1)$-dimensional scalar field theories on the line
\begin{equation} \label{equ:scalar_field_theory}
 (\pt^2 - \px^2) \phi = - W'(\phi), \quad (t,x) \in \bbR \times \bbR,
\end{equation}
where $W \colon \bbR \to [0, \infty)$ is a scalar potential with a double-well, i.e., $W$ has (at least) two consecutive global minima $\phi_-, \phi_+ \in \bbR$ satisfying $W(\phi_\pm) = 0$, and $W''(\phi_\pm) > 0$. A kink solution to~\eqref{equ:scalar_field_theory} is the unique solution (up to symmetries) to
\begin{equation*}
 \left\{ \begin{aligned}
  &\px^2 K = W'(K), \quad x \in \bbR, \\
  &\lim_{x\to\pm\infty} K(x) = \phi_\pm.
 \end{aligned} \right. 
\end{equation*}
Prime examples include the $\phi^4$ model, the more general $P(\phi)_2$ theories, the sine-Gordon model, and the double sine-Gordon theories.

\subsection{Previous works}

The study of the dynamics of solitons in nonlinear dispersive and hyperbolic equations is a rich and vast subject. 
In this subsection we limit ourselves to providing a brief overview of references related to the (conditional) asymptotic stability of solitons in one-dimensional wave-type models.

General results on the decay and the asymptotics of small solutions to one-dimensional Klein-Gordon equations were obtained in 
\cite{Del01, LS05_1, LS05_2, LS06, HN08, HN12, LS15, Sterb16, LLS1, LLS2, LLSS, GP20, GermPusZhang22}.
Regarding the (conditional) asymptotic stability of solitons in 1D wave-type models one distinguishes local asymptotic stability results (in the sense of decay in a spatially localized energy norm) and full asymptotic stability results (in the sense of $L^\infty$ decay estimates and possibly asymptotics). 
For local asymptotic stability results for 1D wave-type models we refer to
\cite{KMM17, KMM17_short, KMM19, KMMV20, KM22, Snelson18}.
Full asymptotic stability results for kink solutions to~\eqref{equ:scalar_field_theory} or soliton solutions to~\eqref{equ:focusing_KG} were obtained in 
\cite{BizTadSzp11, KNS12, KMM19, KK11_1, KK11_2, DelMas20, GP20, CLL20, LS1, GermPusZhang22}.
Pioneering works on radiation damping in the presence of internal modes include~\cite{SofWein99, Sigal93}. 
For further developments, see for instance~\cite{BamCucc11, TsaiYau02, KK11_1, CuccMaeda21_AnnPDE, DelMas20, LegPus21} and references therein.

Finally, we mention the monographs 
\cite{DauxPey10, Lamb80, MantSut04, SG_Series, phi4_Series}
for background on solitons, and we refer to the reviews~\cite{Tao09, KMM17_Survey, CuccMaeda20_Survey, Martel_ICM} as well as to the sample of recent works 
\cite{GP20, DelMas20, CLL20, KMMV20, LLSS, LS1, Chen21, CuccMaeda21, Martel21, LegPus21, KM22}
for further references and perspectives on the study of the asymptotic stability of solitons, or solitary waves, for 1D wave-type and 1D Schr\"odinger models.

\subsection{Overview of the proof}

The proof of Theorem~\ref{thm:codim_asymptotic_stability} is based on several virial-type estimates, combining techniques from the series of works~\cite{KMM17, KMM17_short, KMM19, KMMV20, KM22}. In particular, we largely follow the framework of~\cite{KM22}.

We consider global even perturbations $\bm{\varphi}(t) = \bm{\phi}(t) - \bm{Q}$ of the soliton satisfying the stability condition~\eqref{equ:thm_assumption_closeness}. Correspondingly, we enact the spectral decomposition
\begin{equation} \label{equ:spectral_decomposition}
 \left\{ \begin{aligned}
  \phi(t,x) - Q(x) &= a_1(t) Y_0(x) + z_1(t) Y_2(x) + u_1(t,x), \\
  \pt \phi(t,x) &= \nu a_2(t) Y_0(x) + \mu z_2(t) Y_2(x) + u_2(t,x),
 \end{aligned} \right.
\end{equation}
where we set
\begin{equation*}
 \begin{aligned}
  a_1(t) := \langle Y_0, \phi(t) - Q \rangle, \quad a_2(t) := \nu^{-1} \langle Y_0, \pt \phi(t) \rangle,
 \end{aligned}
\end{equation*}
and
\begin{equation*}
 \begin{aligned}
  z_1(t) := \langle Y_2, \phi(t) - Q \rangle, \quad z_2(t) := \mu^{-1} \langle Y_2, \pt \phi(t) \rangle.
 \end{aligned}
\end{equation*}
Then we have the orthogonality conditions
\begin{equation*}
 \begin{aligned}
  \langle Y_0, u_1(t) \rangle = \langle Y_0, u_2(t) \rangle = 0, \quad \langle Y_2, u_1(t) \rangle = \langle Y_2, u_2(t) \rangle = 0, \quad \text{for all } \, t \geq 0.
 \end{aligned}
\end{equation*}
We write $\bm{u} = (u_1, u_2)$, $\bm{z} = (z_1, z_2)$, and $|\bm{z}|^2 = z_1^2 + z_2^2$. Moreover, we mainly work with the following variables related to the unstable mode
\begin{equation*}
 \begin{aligned}
  b_+ := \frac12 (a_1 + a_2), \quad b_- := \frac12 (a_1 - a_2).
 \end{aligned}
\end{equation*}
Note that $a_1 = b_+ + b_-$ and $a_2 = b_+ - b_-$.
From~\eqref{equ:linearized_KG_system} and the spectral decomposition~\eqref{equ:spectral_decomposition}, we obtain a system of evolution equations for the variables $(u_1, u_2, z_1, z_2, b_+, b_-)$ given by
\begin{equation} \label{equ:pdeode_system_u}
 \left\{ \begin{aligned}
  \pt u_1 &= u_2, \\
  \pt u_2 &= - L u_1 + N^\perp, \\
  \pt z_1 &= \mu z_2, \\
  \pt z_2 &= -\mu z_1 + \mu^{-1} N_2, \\
  \pt b_+ &= \nu b_+ + (2\nu)^{-1} N_0, \\
  \pt b_- &= -\nu b_- - (2\nu)^{-1} N_0, 
 \end{aligned} \right.
\end{equation}
where we use the notation
\begin{equation*}
 \begin{aligned}
  N := \bigl( a_1 Y_0 + z_1 Y_2 + u_1 \bigr)^2, \quad
  N^\perp := N - N_0 Y_0 - N_2 Y_2, \quad
  N_0 := \langle Y_0, N \rangle, \quad
  N_2 := \langle Y_2, N \rangle.
 \end{aligned}
\end{equation*}
By the stability hypothesis~\eqref{equ:thm_assumption_closeness} and the decomposition~\eqref{equ:spectral_decomposition}, we have for all $t \geq 0$ that
\begin{equation} \label{equ:smallness}
 \|u_1(t)\|_{H^1} + \|u_2(t)\|_{L^2} + |z_1(t)| + |z_2(t)| + |a_1(t)| + |a_2(t)| + |b_+(t)| + |b_-(t)| \leq C \delta.
\end{equation}

In Section~\ref{sec:virial_large_scale} we use a first virial argument to obtain integrated-in-time localized estimates for the variable $\bm{u}$ at a large scale. In Section~\ref{sec:internal_mode} we verify a Fermi Golden Rule condition for the model, and we adopt a new functional introduced in~\cite[Proposition 2]{KM22} to gain integrated-in-time control of the internal mode component $\bm{z}$. In Section~\ref{sec:unstable_mode} we bound the variables $(b_+, b_-)$ related to the unstable mode. 
Then we exploit a specific conjugation identity for the linearized operator $L$.
Introducing the first-order differential operators
\begin{equation*} 
 \calD_\ell := \px + \frac{\ell}{2} \tanh\Bigl(\frac{x}{2}\Bigr), \quad 1 \leq \ell \leq 3,
\end{equation*}
it holds 
\begin{equation} \label{equ:conjugation_identity}
 \calD_1 \calD_2 \calD_3 L = (-\px^2 + 1) \calD_1 \calD_2 \calD_3.
\end{equation}
Note that the conjugate operator is just a flat Klein-Gordon operator. 
The identity~\eqref{equ:conjugation_identity} is related to factorizations of Schr\"odinger operators associated with $L$ into first-order differential operators.
Such factorization techniques go back to the classical work of Darboux~\cite{Darboux1882}
and are for instance widely used in quantum mechanics~\cite{InfeldHull51}.
We refer to~\cite[Section 3]{DeiTru}, \cite[Section 3]{CGNT07}, \cite{MatveevSalle_Book} for more background 
and to
\cite{RodSterb10, RaphRod12, KriegerMiao20, KrMiSchl20, KMM19, KMMV20, LS1, Martel21, CuccMaeda21} 
for applications in the study of nonlinear dispersive and hyperbolic equations. 
Following~\cite{KM22}, we call $\calD_1 \calD_2 \calD_3$ the iterated Darboux transformation in this paper.
In Section~\ref{sec:darboux_transformation} we establish several technical estimates for a regularized version of the iterated Darboux transformation.
In Section~\ref{sec:virial_transformed_equation} we then pass to the transformed equation for a regularized version of the variable $\calD_1 \calD_2 \calD_3 \bm{u}$, for which we derive bounds via a second localized virial argument. 
The latter hinges on the fact that the original variable $\bm{u}$ is even, whence $\calD_1 \calD_2 \calD_3 \bm{u}$ is odd, and therefore orthogonal to the even threshold resonance $1$ of the flat Klein-Gordon operator of the transformed equation.
The main observation here is that the transformed equation for the (regularized) variable $\calD_1 \calD_2 \calD_3 \bm{u}$ features the flat Klein-Gordon operator and that one can derive integrated local energy decay for odd solutions to the flat Klein-Gordon equation via a classical positive commutator argument, which is effectively carried out in Section~\ref{sec:virial_transformed_equation} for the regularized version of the variable $\calD_1 \calD_2 \calD_3 \bm{u}$. The bounds on the latter can then be transferred back to the original variable $\bm{u}$, using the integrated-in-time localized estimates for the variable $\bm{u}$ previously established in Section~\ref{sec:virial_large_scale}. This strategy based on two virial arguments was introduced in~\cite{KMM19}.
Finally, in Section~\ref{sec:conclusion} we combine all estimates and conclude the proof of Theorem~\ref{thm:codim_asymptotic_stability}.

\subsection{Notation and preliminaries}

We denote by $C > 0$ an absolute constant whose value may change from line to line. For non-negative $X, Y$ we write $X \lesssim Y$ if $X \leq C Y$, and we use the notation $X \ll Y$ to indicate that the implicit constant should be regarded as small.
Moreover,~we~write
\begin{equation*}
 \begin{aligned}
  \langle f, g \rangle = \int_{\bbR} f(x) g(x) \, \ud x, \quad \|f\| = \sqrt{\langle f, f \rangle}.
 \end{aligned}
\end{equation*}
Our convention for the Fourier transform is
\begin{equation*} 
 \begin{aligned}
  \calF[f](\xi) &= \hat{f}(\xi) = \frac{1}{\sqrt{2\pi}} \int_{\bbR} e^{-ix\xi} f(x) \, \ud x.
 \end{aligned}
\end{equation*}
We denote by $P_c$ the projection onto the continuous spectral subspace of $L^2$ relative to the linearized operator $L$, i.e.,
\begin{equation*}
 P_c f = f - \langle Y_0, f \rangle Y_0 - \langle Y_1, f \rangle Y_1 - \langle Y_2, f \rangle Y_2. 
\end{equation*}
Observe that for even functions $f$, we just have $P_c f = f - \langle Y_0, f \rangle Y_0 - \langle Y_2, f \rangle Y_2$.

\medskip 

\noindent {\it Parameters.}
We use the parameters $A$, $\varepsilon$, $\delta$ in this paper.
In the proof of Theorem~\ref{thm:codim_asymptotic_stability}, we first fix $0 < \varepsilon \ll 1$ sufficiently small, depending on absolute constants, then we choose $A \gg 1$ sufficiently large, depending on $\varepsilon$. Finally, we fix $0 < \delta \ll 1$ sufficiently small depending on $\varepsilon$ and~$A$.

\medskip 

\noindent {\it Notation for the iterated Darboux transformation.} 
We use a regularization in terms of the operator 
\begin{equation*}
 X_\varepsilon = (1-\varepsilon\px^2)^{-2} \colon L^2(\bbR) \to H^4(\bbR), \quad \varepsilon > 0,
\end{equation*}
defined via its Fourier representation
\begin{equation*}
 \widehat{X_\varepsilon f}(\xi) = \frac{\hatf(\xi)}{(1+\varepsilon\xi^2)^2}.
\end{equation*}
Then we define the regularized iterated Darboux transformation by
\begin{equation} \label{equ:definition_regularized_iterated_DT}
  S_\varepsilon := X_\varepsilon \calD_1 \calD_2 \calD_3.
\end{equation}
The adjoints of the differential operators $\calD_\ell$, $1 \leq \ell \leq 3$, with respect to the inner product $\langle \cdot, \cdot \rangle$ are denoted by $\calD_\ell^\ast$, $1 \leq \ell \leq 3$.

\medskip 

\noindent {\it Notation for virial arguments.}
We set 
\begin{equation*}
 \begin{aligned}
  \rho(x) = \sech^2\Bigl(\frac{x}{20}\Bigr).
 \end{aligned}
\end{equation*}
Let $\chi \in C_c^\infty(\bbR)$ be a smooth even bump function satisfying
\begin{equation*}
 \begin{aligned}
  \chi(x) = 1 \quad \text{for } |x| \leq 1, \quad \chi(x) = 0 \quad \text{for } |x| \geq 2, \quad \chi'(x) \leq 0 \quad \text{for } x \geq 0.
 \end{aligned}
\end{equation*}
For the first virial estimate at large scale in Section~\ref{sec:virial_large_scale} we use
\begin{equation*}
 \begin{aligned}
  \zeta_A(x) = \exp \Bigl( -\frac{1}{A} \bigl(1-\chi(x)\bigr) |x| \Bigr), \quad \Phi_A(x) = \int_0^x \zeta_A(y)^2 \, \ud y,
 \end{aligned}
\end{equation*}
while for the second virial estimate for the transformed equation in Section~\ref{sec:virial_transformed_equation} we use
\begin{equation*}
 \begin{aligned}
  &\zeta_B(x) = \sech\Bigl(\frac{x}{B}\Bigr), \quad \Phi_B(x) = B \tanh\Bigl(\frac{x}{B}\Bigr), \quad B = 100, \\
  &\quad \quad \quad \Psi_{A, B}(x) = \chi_A(x)^2 \Phi_B(x), \quad \chi_A(x) = \chi\Bigl(\frac{x}{A}\Bigr).
 \end{aligned}
\end{equation*}
The choice of the switch function $\Phi_B$ is specifically adapted to the fact that the linear operator in the transformed equation~\eqref{equ:transformed_equation} is just the flat Klein-Gordon operator. 
Moreover, we introduce the weight function
\begin{equation*}
 \begin{aligned}
  \sigma_A(x) = \sech\Bigl( \frac{2}{A} x \Bigr).
 \end{aligned}
\end{equation*}
We use that $\zeta_A^2 \lesssim \sigma_A \lesssim \zeta_A^2$.
As in \cite{KM22} it is convenient to introduce a space $\calY$ of smooth functions $f \colon \bbR \to \bbR$ with the property that for any $k \geq 0$, there exists a constant $C_k > 0$ such that 
\begin{equation*}
 |f^{(k)}(x)| \leq C_k \rho(x)^3 \quad \text{for all } x \in \bbR. 
\end{equation*}
Observe that $Q, Y_0, Y_1, Y_2 \in \calY$.

\medskip 

\noindent {\it General virial identities.} 
Let $\Phi \colon \bbR \to \bbR$ be a smooth, odd, strictly increasing, bounded function and let $\bm{w} = (w_1, w_2)$ be an $H^1 \times L^2$ solution to
\begin{equation*}
 \left\{ \begin{aligned}
  \pt w_1 &= w_2, \\
  \pt w_2 &= - L w_1 + G,
 \end{aligned} \right.
\end{equation*}
where $L = -\px^2 + P$ and $G = G(t,x)$, $P = P(x)$ are given functions. We define 
\begin{equation*}
 \calA = \int \biggl( \Phi \px w_1 + \frac12 \Phi' w_1 \biggr) w_2.
\end{equation*}
Then we have 
\begin{equation} \label{equ:prelim_general_virial_identity}
 \begin{aligned}
  \pt \calA &= - \int \Phi' (\px w_1)^2 + \frac{1}{4} \int \Phi''' w_1^2 + \frac12 \int \Phi P' w_1^2 + \int G \biggl( \Phi \px w_1 + \frac12 \Phi' w_1 \biggr).
 \end{aligned}
\end{equation}
Moreover, introducing the variables 
\begin{equation*}
 \tilde{w}_1 = \zeta w_1, \quad \zeta = \sqrt{\Phi'},
\end{equation*}
the identity~\eqref{equ:prelim_general_virial_identity} can be rewritten as
\begin{equation} \label{equ:prelim_general_virial_identity_alternative}
 \begin{aligned}
  \pt \calA = - \int (\px \tilde{w}_1)^2 - \frac12 \int \biggl( \frac{\zeta''}{\zeta} - \frac{(\zeta')^2}{\zeta^2} \biggr) \tilde{w}_1^2 + \frac12 \int \frac{\Phi P'}{\zeta^2} \tilde{w}_1^2 + \int G \biggl( \Phi \px w_1 + \frac12 \Phi' w_1 \biggr).
 \end{aligned}
\end{equation}
See the proof of \cite[Proposition 1]{KMM19} for detailed computations.

\section{Virial estimate at large scale} \label{sec:virial_large_scale}

In this section we obtain preliminary estimates for $\bm{u}$ via a first virial argument.
\begin{proposition} \label{prop:virial_u}
 For any $A > 0$ large, any $\delta > 0$ small (depending on $A$), and any $T > 0$, 
 \begin{equation*}
  \int_0^T \biggl( \|\sigma_A \px u_1\|^2 + \frac{1}{A^2} \|\sigma_A u_1\|^2 + \frac{1}{A^2} \|\sigma_A u_2\|^2 \biggr) \, \ud t \lesssim A \delta^2 + \int_0^T \bigl( \|\rho u_1\|^2 + |\bm{z}|^4 + b_+^2 + b_-^2 \bigr) \, \ud t.
 \end{equation*}
\end{proposition}
\begin{proof}
We proceed similarly to the proof of~\cite[Proposition 1]{KM22} and use the virial functional 
\begin{equation*}
 \begin{aligned}
  \calI := \int \biggl( \Phi_A \px u_1 + \frac12 \Phi_A' u_1 \biggr) u_2.  
 \end{aligned}
\end{equation*}
Denote by $\tilu_1 = \zeta_Au_1$. By \eqref{equ:pdeode_system_u} and the general virial identity \eqref{equ:prelim_general_virial_identity_alternative}, we have
\begin{equation*}
\begin{split}
 \pt \calI &= -\int (\partial_x \tilu_1)^2 - \frac{1}{2} \int \left(\frac{\zeta_A''}{\zeta_A} - \frac{(\zeta_A')^2}{\zeta_A^2}\right) \tilu_1^2 - \int \frac{\Phi_A}{\zeta_A^2} Q' \tilu_1^2 + \int N^\perp \biggl(\Phi_A \partial_x u_1 + \frac{1}{2}\Phi_A'u_1\biggr) \\
 &= -\Vert \partial_x \tilu_1 \Vert^2 + I_1 + I_2 + I_3.
\end{split}
\end{equation*}
As in the proof of \cite[Lemma 1]{KMM19}, we obtain
\begin{equation*}
 \left \vert \frac{\zeta_A''}{\zeta_A} - \frac{(\zeta_A')^2}{\zeta_A} \right \vert \lesssim \frac{\rho^2}{A},
\end{equation*}
whence
\begin{equation*}
 |I_1| \lesssim \frac{1}{A}\Vert \rho \tilu_1\Vert^2 \lesssim \Vert \rho u_1 \Vert^2.
\end{equation*}
Since $Q' \in \calY$, $\zeta_A \geq e^{-\frac{\vert x \vert}{A}}$, and $\vert \Phi_A \vert \leq \vert x \vert$,
we have for $A$ large,
\begin{equation*}
 \left \vert \frac{\Phi_A}{\zeta_A^2} Q' \right \vert \lesssim \vert x \vert e^{\frac{2}{A}\vert x \vert} \rho^3 \lesssim \rho^2,
\end{equation*}
and thus,
\begin{equation*}
 |I_2| \lesssim \Vert\rho \tilu_1\Vert^2 \lesssim \Vert \rho u_1 \Vert^2. 
\end{equation*}
To estimate the last term $I_3$, we decompose the nonlinearity $N^\perp$ into several parts and write 
\begin{equation*}
 \begin{aligned}
  I_3 &= \int (a_1 Y_0 + z_1 Y_2)^2 \biggl(\Phi_A \partial_x u_1 + \frac{1}{2}\Phi_A'u_1\biggr) + \int 2 (a_1 Y_0 + z_1 Y_2) u_1 \biggl(\Phi_A \partial_x u_1 + \frac{1}{2}\Phi_A'u_1\biggr) \\
  &\quad + \int u_1^2 \biggl(\Phi_A \partial_x u_1 + \frac{1}{2}\Phi_A'u_1\biggr) + \int \bigl( -N_0 Y_0 - N_2 Y_2 \bigr) \biggl(\Phi_A \partial_x u_1 + \frac{1}{2}\Phi_A'u_1\biggr) \\
  &= I_{3,1} + I_{3,2} + I_{3,3} + I_{3,4}.
 \end{aligned}
\end{equation*}
In what follows we use on several occasions that for any function $F \in \calY$, we obtain by integration by parts and the Cauchy-Schwarz inequality, using that $|\Phi_A| \leq |x|$, $|\Phi'_A| \leq 1$,
\begin{equation} \label{equ:prop_first_virial_repeated_intbyparts_identity}
 \begin{aligned}
  \biggl| \int F \biggl(\Phi_A \partial_x u_1 + \frac{1}{2}\Phi_A'u_1\biggr) \biggr| = \biggl| \int \biggl( F' \Phi_A + \frac12 F \Phi_A' \biggr) u_1 \biggr| \lesssim \int \rho^3 (|x| + 1) |u_1| \lesssim \|\rho u_1\|.
 \end{aligned}
\end{equation}
Using~\eqref{equ:prop_first_virial_repeated_intbyparts_identity} and \eqref{equ:smallness}, the first term $I_{3,1}$ can then be bounded by
\begin{equation*}
 \begin{aligned}
  |I_{3,1}| &\lesssim \bigl( a_1^2 + |a_1| |z_1| + |z_1|^2 \bigr) \|\rho u_1\| \lesssim \|\rho u_1\|^2 + |\bm{z}|^4 + b_+^2 + b_-^2. 
 \end{aligned}
\end{equation*}
To estimate the second term $I_{3,2}$, we integrate by parts and use $Y_0, Y_2 \in \calY$, $|\Phi_A(x)| \leq |x|$, as well as \eqref{equ:smallness} to find 
\begin{equation*}
 \begin{aligned}
  |I_{3,2}| = \biggl| \int \bigl( a_1 Y_0' + z_1 Y_2' \bigr) \Phi_A u_1^2 \biggr| \lesssim \bigl( |a_1| + |z_1| \bigr) \|\rho u_1\|^2 \lesssim \|\rho u_1\|^2.
 \end{aligned}
\end{equation*}
Integrating by parts in the third term $I_{3,3}$, using $\Phi_A' = \zeta_A^2$, and~\eqref{equ:smallness}, we obtain
\begin{equation*}
 I_{3,3} = \biggl| \frac16 \int \Phi_A' u_1^3 \biggr| \lesssim \int \zeta_A^2 |u_1|^3 \lesssim A^2 \|u_1\|_{L^\infty} \|\px \tilu_1\|^2 \lesssim A^2 \delta \|\px \tilde{u}_1\|^2,
\end{equation*}
where in the second to last estimate we made use of~\cite[Claim 1]{KMM19}.
For the last term $I_{3,4}$ we use the pointwise estimate $|N| \lesssim u_1^2 + |\bm{z}|^2 + b_+^2 + b_-^2$, the fact that $Y_0, Y_2 \in \calY$, estimate~\eqref{equ:prop_first_virial_repeated_intbyparts_identity}, and~\eqref{equ:smallness} to conclude 
\begin{equation*}
 \begin{aligned}
  |I_{3,4}| \lesssim \bigl( \|\rho u_1\|^2 + |\bm{z}|^2 + b_+^2 + b_-^2 \bigr) \|\rho u_1\| \lesssim \|\rho u_1\|^2 + |\bm{z}|^4 + b_+^2 + b_-^2. 
 \end{aligned}
\end{equation*}

Gathering the preceding estimates and taking $\delta$ small enough depending on $A$ yields
\begin{equation}\label{eqn: dot I functional}
 \pt \calI \leq -\frac{1}{2}\Vert \partial_x \tilu_1\Vert^2 + C \bigl( \| \rho u_1 \|^2 + |\bm{z}|^4 + b_+^2 + b_-^2 \bigr).
\end{equation}
Moreover, mimicking the proof of estimate (19) in \cite{KM22}, we have 
\begin{equation} \label{equ:prop_first_virial_est19_from_KM22}
 \Vert \sigma_A \partial_x u_1\Vert^2 + \frac{1}{A^2} \Vert \sigma_A u_1 \Vert^2 \lesssim \Vert \partial_x \tilu_1\Vert^2 + \frac{1}{A}\Vert \rho u_1 \Vert^2.
\end{equation}
Thus, upon integrating~\eqref{eqn: dot I functional} in time and using~\eqref{equ:prop_first_virial_est19_from_KM22} as well as $\vert \calI \vert \lesssim A\delta^2$ by~\eqref{equ:smallness}, we obtain for any $T > 0$ that
\begin{equation}\label{eqn:estimate from calI}
 \int_0^T \Bigl( \Vert \sigma_A\partial_x u_1\Vert^2 + \frac{1}{A^2} \Vert \sigma_Au_1\Vert^2 \Bigr) \, \ud t \lesssim A\delta^2 + \int_0^T \bigl( \Vert \rho u_1 \Vert^2 + \vert \bm{z}\vert^4 + b_+^2 + b_-^2 \bigr) \, \ud t.
\end{equation}

To complete the proof of the proposition, we introduce the functional
\begin{equation*}
 \calH := \int \sigma_A^2 u_1 u_2.
\end{equation*}
Using \eqref{equ:pdeode_system_u} and integrating by parts, we compute
\begin{equation*}
\begin{split}
 \pt \calH &= \int \sigma_A^2 u_2^2 - \int \sigma_A^2(\partial_x u_1)^2 + \frac{1}{2}\int (\sigma_A^2)''u_1^2 + \int \sigma_A^2 (2Q-1)u_1^2  + \int \sigma_A^2 u_1 N^\perp \\
 &= \Vert \sigma_A u_2 \Vert^2 - \Vert \sigma_A \partial_x u_1 \Vert^2 + H_1 + H_2 + H_3.
\end{split}
\end{equation*}
Since $\vert (\sigma_A^2)''\vert  \lesssim \sigma_A^2$ and $|2Q-1| \lesssim 1$, we have
\begin{equation*}
 |H_1| + |H_2| \lesssim \Vert \sigma_A u_1 \Vert^2.
\end{equation*}
Using the pointwise estimate $\vert N \vert \lesssim u_1^2 + a_1^2 \rho^6 + z_1^2 \rho^6$ and~\eqref{equ:smallness}, we find
\begin{equation*}
 |H_3| \lesssim \Vert \sigma_A u_1 \Vert^2 + \vert \bm{z} \vert^4 + b_+^2 + b_-^2.
\end{equation*}
Thus, we arrive at the inequality
\begin{equation*}
 \pt \calH \geq \Vert \sigma_A u_2 \Vert^2 - \Vert \sigma_A \partial_x u_1 \Vert^2 - C(\Vert \sigma_A u_1 \Vert^2 + \vert \bm{z} \vert^4 + b_+^2 + b_-^2).
\end{equation*}
Integrating in time and using $\vert \calH \vert \lesssim \Vert u_1 \Vert \Vert u_2 \Vert \lesssim \delta^2$ by \eqref{equ:smallness}, we obtain for any $T > 0$,
\begin{equation}\label{eqn:estimate from calH}
 \frac{1}{A^2}\int_0^T \Vert \sigma_A u_2 \Vert^2 \, \ud t \lesssim \delta^2 + \frac{1}{A^2} \int_0^T \left(\Vert \sigma_A \partial_x u_1 \Vert^2 + \Vert \sigma_A u_1 \Vert^2 + \vert \bm{z} \vert^4 + b_+^2 + b_-^2 \right)\ud t.
\end{equation}
Combining \eqref{eqn:estimate from calI} and \eqref{eqn:estimate from calH} finishes the proof of the proposition.
\end{proof}

\section{Controlling the internal mode} \label{sec:internal_mode}

We first verify a natural Fermi Golden Rule condition for perturbations of the soliton of the one-dimensional quadratic Klein-Gordon equation.

\begin{lemma}[Fermi Golden Rule] \label{lem:FGR}
 The smooth bounded function 
 \begin{equation} \label{equ:FGRg}
  \begin{aligned}
   g(x) &:= \cos \bigl(\sqrt{2} x\bigr) \biggl( -\frac{3}{2 \sqrt{2}} + \frac{15}{2 \sqrt{2}} \sech^2\left(\frac{x}{2}\right) \biggr) \\ 
   &\quad \quad + \sin \bigl(\sqrt{2} x \bigr) \biggl( - \frac{57}{8} \tanh \left(\frac{x}{2}\right) + \frac{15}{8} \tanh ^3\left(\frac{x}{2}\right) \biggr)
  \end{aligned}
 \end{equation}
 satisfies $L g = 4\mu^2 g$ and
 \begin{equation} \label{equ:FGR}
  \Gamma := \frac12 \int Y_2^2 g = \frac{243}{32} \pi  \cosech\bigl(\sqrt{2} \pi \bigr) \neq 0.
 \end{equation}
\end{lemma}
\begin{proof}
Recall that $\mu^2 = \frac34$.
The function $h(x) = \sin(\sqrt{2}x)$ clearly satisfies $(-\px^2+1) h = 4\mu^2 h$. Define the smooth bounded function $g := \calD_3^\ast \calD_2^\ast \calD_1^\ast h$, which is given by~\eqref{equ:FGRg}. 
Using the adjoint form of the conjugation identity~\eqref{equ:conjugation_identity}, we find that 
\begin{equation*}
 L g = L \calD_3^\ast \calD_2^\ast \calD_1^\ast h = \calD_3^\ast \calD_2^\ast \calD_1^\ast (-\px^2 +1) h = 4\mu^2 g,
\end{equation*}
as desired. Next, we compute 
\begin{equation*}
 \begin{aligned}
  \int Y_2^2 g &= \int Y_2^2 \calD_3^\ast \calD_2^\ast \calD_1^\ast h = \int \calD_1 \calD_2 \calD_3 (Y_2^2) \sin\bigl(\sqrt{2} x\bigr).
 \end{aligned}
\end{equation*}
Let $H = \calD_1 \calD_2 \calD_3 (Y_2^2)$. Since $Y_2$ is even, $H$ is odd, and thus,
\begin{equation} \label{equ:FGR_integralY2sqg_widehatH}
 \begin{aligned}
  \frac12 \int Y_2^2 g = i \sqrt{\frac{\pi}{2}} \, \widehat{H}(\sqrt{2}).
 \end{aligned}
\end{equation}
By patient direct computation, we obtain 
\begin{equation} \label{equ:patient_comp1}
\begin{split}
 H(x) &= \frac{9}{256} \biggl( 875\sech^8\Bigl(\frac{x}{2}\Bigr) - 700\sech^6\Bigl(\frac{x}{2}\Bigr) + 64 \sech^4\Bigl(\frac{x}{2}\Bigr) \biggr)\tanh\Bigl(\frac{x}{2}\Bigr).
\end{split}
\end{equation}
Moreover, by patient direct computation we find
\begin{equation} \label{equ:patient_comp2}
 H(x) =  \frac{1}{256} \bigl( 28\partial_x + 17\partial_x^3 -70\partial_x^5 + 25\partial_x^7 \bigr) \Bigl[ \sech^2\Bigl(\frac{x}{2}\Bigr) \Bigr].
\end{equation}
Using that (cf. \cite[Corollary 5.7]{LS1})
\begin{equation*}
 \mathcal{F}\Bigl[ \sech^2\Bigl(\frac{\cdot}{2}\Bigr) \Bigr](\xi) = \sqrt{8\pi} \xi \cosech(\pi \xi), 
\end{equation*}
we conclude
\begin{equation} \label{equ:FGR_widehatH}
\begin{split}
 \widehat{H}(\xi) &= - \frac{i}{64} \sqrt{\frac{\pi}{2}} \bigl(-28 + 17 \xi^2 + 70 \xi^4 + 25 \xi^6\bigr)  \xi^2 \cosech(\pi \xi).	
\end{split}
\end{equation}
Combining~\eqref{equ:FGR_integralY2sqg_widehatH} and \eqref{equ:FGR_widehatH} gives~\eqref{equ:FGR}.
\end{proof}

\begin{remark}
 We determined the identities~\eqref{equ:patient_comp1} and \eqref{equ:patient_comp2} with the aid of the Wolfram Mathematica software system.
\end{remark}

Next, we use a new functional introduced in~\cite[Proposition 2]{KM22} to obtain integrated-in-time control of the internal mode component $\bm{z}$.

\begin{proposition} \label{prop:virial_z}
 For any $A > 0$ large, any $\delta > 0$ small (depending on $A$), and any $T > 0$, 
 \begin{equation*}
  \int_0^T |\bm{z}|^4 \, \ud t \lesssim A \delta^2 + \frac{1}{\sqrt{A}} \int_0^T \Bigl( \|\sigma_A \px u_1\|^2 + \frac{1}{A^2} \|\sigma_A u_1\|^2 + \frac{1}{A^2} \|\sigma_A u_2\|^2 + b_+^2 + b_-^2 \Bigr) \, \ud t.
 \end{equation*}
\end{proposition}
\begin{proof}
We define the variables
\begin{align*}
 \alpha := z_1^2 - z_2^2, \quad \beta := 2 z_1 z_2.
\end{align*}
From \eqref{equ:pdeode_system_u} we compute 
\begin{equation} \label{equ:prop_internal_mode_equations_alphabeta}
 \begin{aligned}
  \pt \alpha &= 2\mu\beta - 2\mu^{-1} z_2 N_2, \\
  \pt \beta &= -2\mu\alpha + 2\mu^{-1} z_1 N_2,
 \end{aligned}
\end{equation}
and
\begin{equation} \label{equ:pt_modzsquared}
  \pt \bigl( |\bm{z}|^2 \bigr) = 2\mu^{-1} z_2 N_2.
\end{equation}
We use the functional 
\begin{equation}
 \calJ := \alpha \int u_2 g \chi_A - 2\mu\beta \int u_1 g \chi_A + \frac{\Gamma}{2\mu} \beta |\bm{z}|^2,
\end{equation}
where $g$ and $\Gamma$ are furnished by Lemma~\ref{lem:FGR}.
By direct computation, using~\eqref{equ:pdeode_system_u} and \eqref{equ:prop_internal_mode_equations_alphabeta}, we find 
\begin{equation*}
 \begin{aligned}
  \pt \calJ &= - \alpha \int (L - 4\mu^2) u_1 g \chi_A - \alpha \biggl( \Gamma |\bm{z}|^2 - \int N^\perp g \chi_A \biggr) \\
  &\quad - 2\mu^{-1} N_2 \biggl( z_2 \int u_2 g \chi_A + 2\mu z_1 \int u_1 g \chi_A \biggr) + \mu^{-2} \Gamma N_2 \bigl( z_1 |\bm{z}|^2 + z_2 \beta \bigr) \\
  &= J_1 + J_2 + J_3 + J_4.
 \end{aligned}
\end{equation*}
For the first term $J_1$ we integrate by parts and exploit that $(L-4\mu^2)g = 0$, to rewrite it as
\begin{equation*}
 J_1 = - \alpha \biggl( 2 \int (\px u_1) g \chi_A' + \int u_1 g \chi_A'' \biggr).
\end{equation*}
Using $|g| \lesssim 1$ and $\sigma_A \gtrsim 1$ on $[-2A, 2A]$, we obtain by Cauchy-Schwarz 
\begin{equation} \label{equ:prop_internal_mode_J1_bound}
 |J_1| \lesssim \frac{1}{\sqrt{A}} \Bigl( \|\sigma_A \px u_1\|^2 + \frac{1}{A^2} \|\sigma_A u_1\|^2 + |\bm{z}|^4 \Bigr).
\end{equation}

Next, we turn to the term $J_2$ and decompose it as
\begin{equation*}
 \begin{aligned}
  J_2 &= -\alpha \biggl( \Gamma |\bm{z}|^2 - z_1^2 \int P_c(Y_2^2) g \chi_A \biggr) + \alpha \int \bigl( N^\perp - z_1^2 P_c(Y_2^2) \bigr) g \chi_A = J_{2, 1} + J_{2,2}.
 \end{aligned}
\end{equation*}
For the analysis of $J_{2,1}$ observe that $\langle Y_0, g \rangle = \langle Y_2, g \rangle = 0$ in view of $L g = 4\mu^2 g$, $L Y_0 = -\nu^2 Y_0$, and $L Y_2 = \mu^2 Y_2$.
Recalling~\eqref{equ:FGR} and that $P_c(Y_2^2) = Y_2^2 - \langle Y_0, Y_2^2 \rangle Y_0 - \langle Y_2, Y_2^2 \rangle Y_2$, we write
\begin{equation} \label{equ:prop_internal_mode_rewrite1}
 \begin{aligned}
  \Gamma |\bm{z}|^2 - z_1^2 \int P_c(Y_2^2) g \chi_A &= \Gamma |\bm{z}|^2 - 2 \Gamma z_1^2 + z_1^2 \int Y_2^2 g (1-\chi_A) \\
  &\quad - z_1^2 \langle Y_0, Y_2^2 \rangle \int Y_0 g (1-\chi_A) - z_1^2 \langle Y_2, Y_2^2 \rangle \int Y_2 g (1-\chi_A).
 \end{aligned}
\end{equation}
Since $Y_2 \in \calY$ and $|g| \lesssim 1$, we have 
\begin{equation*}
 \begin{aligned}
  \biggl| z_1^2 \int Y_2^2 g (1-\chi_A) \biggr| \lesssim |\bm{z}|^2 \int_{|x| \geq A} \rho^6 \lesssim \rho(A) |\bm{z}|^2,
 \end{aligned}
\end{equation*}
with analogous bounds for the last two terms on the right-hand side of~\eqref{equ:prop_internal_mode_rewrite1}. 
Finally, noting that $\Gamma |\bm{z}|^2 - 2 \Gamma z_1^2 = - \Gamma \alpha$, we conclude 
\begin{equation} \label{equ:prop_internal_mode_J21_bound}
 \begin{aligned}
  \bigl| J_{2,1} - \Gamma \alpha^2 \bigr| \lesssim \rho(A) |\bm{z}|^4.
 \end{aligned}
\end{equation}
To estimate the term $J_{2,2}$ we use that $\chi_A \lesssim \sigma_A^2$ and that $Y_0, Y_2 \in \calY$ to obtain by Cauchy-Schwarz
\begin{equation*}
 \begin{aligned}
  |J_{2,2}| &\lesssim |\bm{z}|^2 \int \bigl| N^\perp - z_1^2 P_c(Y_2^2) \bigr| \chi_A \\
  &\lesssim |\bm{z}|^2 \bigl( \|\sigma_A u_1\|^2 + b_+^2 + b_-^2 + \|\sigma_A u_1\| |z_1| + (|b_+| + |b_-|)|z_1| \bigr),
 \end{aligned}
\end{equation*}
whence by~\eqref{equ:smallness} we have 
\begin{equation} \label{equ:prop_internal_mode_J22_bound}
 \begin{aligned}
  |J_{2,2}| \lesssim \delta \bigl( |\bm{z}|^4 + \|\sigma_A u_1\|^2 + b_+^2 + b_-^2 \bigr).
 \end{aligned}
\end{equation}
Combining~\eqref{equ:prop_internal_mode_J21_bound} and \eqref{equ:prop_internal_mode_J22_bound}, we find 
\begin{equation*}
 \bigl| J_2 - \Gamma \alpha^2 \bigr| \lesssim \bigl( \delta + \rho(A) \bigr) |\bm{z}|^4 + \delta \bigl( \|\sigma_A u_1\|^2 + b_+^2 + b_-^2 \bigr).
\end{equation*}

To estimate the term $J_3$ we first use $Y_2 \in \calY$ and the pointwise bound $|N| \lesssim u_1^2 + z_1^2 + b_+^2 + b_-^2$ to deduce
\begin{equation*}
 |N_2| \lesssim \|\sigma_A u_1\|^2 + z_1^2 + b_+^2 + b_-^2.
\end{equation*}
Using that $\sigma_A \gtrsim 1$ on $[-2A, 2A]$, we then obtain by Cauchy-Schwarz
\begin{equation}
 \begin{aligned}
  |J_3| &\lesssim |N_2| |\bm{z}| \sqrt{A} \bigl( \|\sigma_A u_1\| + \|\sigma_A u_2\| \bigr) \\
  &\lesssim \delta |\bm{z}|^4  + \delta A \bigl( \|\sigma_A u_1\|^2 + \|\sigma_A u_2\|^2 + b_+^2 + b_-^2 \bigr) \\
  &\lesssim \delta |\bm{z}|^4  + \frac{1}{\sqrt{A}} \Bigl( \frac{1}{A^2} \|\sigma_A u_1\|^2 + \frac{1}{A^2} \|\sigma_A u_2\|^2 + b_+^2 + b_-^2 \
  \Bigr),
 \end{aligned}
\end{equation}
where we chose $\delta \leq A^{-\frac72}$ to pass to the last line.
Similarly, we obtain
\begin{equation}
 \begin{aligned}
  |J_4| \lesssim |N_2| |\bm{z}|^3 \lesssim \delta |\bm{z}|^4 + \frac{1}{\sqrt{A}} \Bigl( \frac{1}{A^2} \|\sigma_A u_1\|^2 + b_+^2 + b_-^2 \Bigr).
 \end{aligned}
\end{equation}
Thus, for $A$ large enough, we find that
\begin{equation} \label{equ:prop_internal_mode_J_bound}
 \bigl| \pt \calJ - \Gamma \alpha^2 \bigr| \lesssim \frac{1}{\sqrt{A}} \Bigl( \|\sigma_A \px u_1\|^2 + \frac{1}{A^2} \|\sigma_A u_1\|^2 + \frac{1}{A^2} \|\sigma_A u_2\|^2 + b_+^2 + b_-^2 + |\bm{z}|^4 \Bigr).
\end{equation}

Next, define 
\begin{equation*}
 \calZ := \frac{\Gamma}{4\mu} \alpha \beta.
\end{equation*}
Using~\eqref{equ:pdeode_system_u} we compute 
\begin{equation*}
 \begin{aligned}
  \pt \calZ &= \frac{\Gamma}{2} \bigl( \beta^2 - \alpha^2 \bigr) + \frac{\Gamma}{2\mu^2} N_2 \bigl( - \beta z_2 + \alpha z_1 \bigr).
 \end{aligned}
\end{equation*}
The last term can be estimated analogously to the term $J_4$ above. Combining with~\eqref{equ:prop_internal_mode_J_bound} and noting that $|\bm{z}|^4 = \alpha^2 + \beta^2$, we arrive at the estimate 
\begin{equation} \label{equ:prop_internal_mode_JZ_bound}
 \Bigl| \pt \calJ + \pt \calZ - \frac{\Gamma}{2} |\bm{z}|^4 \Bigr| \lesssim \frac{1}{\sqrt{A}} \Bigl( \|\sigma_A \px u_1\|^2 + \frac{1}{A^2} \|\sigma_A u_1\|^2 + \frac{1}{A^2} \|\sigma_A u_2\|^2 + b_+^2 + b_-^2 + |\bm{z}|^4 \Bigr). 
\end{equation}
Finally, by \eqref{equ:smallness} we have $|\calZ| \lesssim \delta^4$ and we infer $|\calJ| \lesssim \sqrt{A} \delta^3$ after an application of Cauchy-Schwarz. The asserted estimate now follows from~\eqref{equ:prop_internal_mode_JZ_bound} upon integrating in time and choosing $A$ sufficiently large (independently of the size of $\delta$).
\end{proof}

\section{Controlling the unstable mode} \label{sec:unstable_mode}

In this section we establish estimates for the variables $b_+$ and $b_-$ related to the unstable mode.

\begin{proposition} \label{prop:virial_b}
For any $\delta > 0$ small and any $T>0$, 
 \begin{equation*} 
  \int_0^T \bigl( b_+^2 + b_-^2 \bigr) \, \ud t \lesssim \delta^2 + \delta \int_0^T \|\rho u_1\|^2 \, \ud t + \int_0^T |\bm{z}|^4 \, \ud t.
 \end{equation*}
\end{proposition}
\begin{proof}
As in \cite[Lemma 8]{KMM19}, we use the functional
\begin{equation*}
 \calB := b_+^2 - b_-^2.
\end{equation*}
Using~\eqref{equ:pdeode_system_u} we compute
\begin{equation} \label{equ:prop_unstable_mode_pt_B}
 \pt \calB = 2\nu (b_+^2 + b_-^2) + \nu^{-1} N_0 (b_+ + b_-).
\end{equation}
Since $|N| \lesssim u_1^2 + z_1^2 + b_+^2 + b_-^2$, we have 
\begin{equation*}
 |N_0| \lesssim \|\rho u_1\|^2 + z_1^2 + b_+^2 + b_-^2.
\end{equation*}
Thus, using~\eqref{equ:smallness} we obtain by Cauchy-Schwarz for any $\gamma > 0$ that
\begin{equation*}
 \begin{aligned}
  |N_0 (b_+ + b_-)| &\lesssim \delta \|\rho u_1\|^2 + \delta (b_+^2 + b_-^2) + z_1^2 (|b_+| + |b_-|) \\
  &\lesssim \delta \|\rho u_1\|^2 + \delta (b_+^2 + b_-^2) + \gamma^{-1} |\bm{z}|^4 + \gamma (b_+^2 + b_-^2).
 \end{aligned}
\end{equation*}
For $\delta > 0$ and $\gamma > 0$ sufficiently small (depending only on absolute constants), the asserted estimate now follows from~\eqref{equ:prop_unstable_mode_pt_B} upon integrating in time.
\end{proof}

\section{Bounds for the iterated Darboux transformation} \label{sec:darboux_transformation}

In this section we establish several bounds for the regularized iterated Darboux transformation~$S_\varepsilon$.
We begin with some preparations, building on computations from~\cite{LS1}, and introduce the auxiliary functions
\begin{equation*}
 Z_\ell(x) := \sech^\ell\Bigl(\frac{x}{2}\Bigr), \quad 1 \leq \ell \leq 3.
\end{equation*}
Observe that $Y_0 = c_0 Z_3$. Then we have $\calD_\ell = Z_\ell \cdot \px \cdot Z_\ell^{-1}$, $1 \leq \ell \leq 3$, and it is evident that $\calD_\ell Z_\ell = 0$, $1 \leq \ell \leq 3$.
Correspondingly, the integral operators
\begin{equation*}
 \begin{aligned}
  \calR_\ell[f](x) &:= Z_\ell(x) \int_0^x Z_\ell(y)^{-1} f(y) \, \ud y, \quad 1 \leq \ell \leq 3,
 \end{aligned}
\end{equation*}
are right-inverse operators for $\calD_\ell$, i.e., $\calD_\ell \calR_\ell[f] = f$, $1 \leq \ell \leq 3$,
and the operator 
\begin{equation*}
 \calR[f] := \calR_3[ \calR_2[ \calR_1[ f ] ] ] 
\end{equation*}
satisfies $\calD_1 \calD_2 \calD_3 \calR[f] = f$.
Integrating by parts, we obtain 
\begin{equation*}
 \begin{aligned}
  \calR_\ell[\calD_\ell f] = f - f(0) Z_\ell, \quad 1 \leq \ell \leq 3,
 \end{aligned}
\end{equation*}
and thus 
\begin{equation*}
 \begin{aligned}
  \calR[\calD_1\calD_2\calD_3 f] = f - f(0) Z_3 - (\calD_3 f)(0) \calR_3[Z_2] - (\calD_2 \calD_3 f)(0) \calR_3[\calR_2[Z_1]].
 \end{aligned}
\end{equation*}
In view of the identities
\begin{equation*}
 \begin{aligned}
  (\calD_3 f)(0) &= f'(0), \quad (\calD_2 \calD_3 f)(0) = f''(0) + \tfrac34 f(0), 
 \end{aligned}
\end{equation*}
and
\begin{equation} \label{equ:RZells_justYjs}
 \begin{aligned}
  Z_3(x) &= \sech^3\Bigl(\frac{x}{2}\Bigr) = c_0^{-1} Y_0(x), \\
  \calR_3[Z_2](x) &= 2 \sech^2\Bigl(\frac{x}{2}\Bigr) \tanh\Bigl(\frac{x}{2}\Bigr) = 2 c_1^{-1} Y_1(x), \\
  \calR_3[\calR_2[Z_1]](x) &= 2 \sech\Bigl(\frac{x}{2}\Bigr) \tanh^2\Bigl(\frac{x}{2}\Bigr) = \tfrac12 c_0^{-1} Y_0(x) - \tfrac12 c_2^{-1} Y_2(x),
 \end{aligned}
\end{equation}
we find 
\begin{equation*}
 \begin{aligned}
  \calR[\calD_1\calD_2\calD_3 f] &= f - c_0^{-1} f(0) Y_0  - 2 c_1^{-1} f'(0) Y_1 - \bigl( f''(0) + {\textstyle \frac34} f(0) \bigr) \bigl( {\textstyle \frac12} c_0^{-1} Y_0 - {\textstyle \frac12} c_2^{-1} Y_2 \bigr).
 \end{aligned}
\end{equation*}
In particular, it follows that
\begin{equation} \label{equ:PcRDsg_equal_Pcg}
 P_c \calR[\calD_1\calD_2\calD_3 f] = P_c f.
\end{equation}

\medskip 

Next, we present a few technical estimates for the smoothing operator $X_\varepsilon$, which are needed below in the proofs of the main bounds for the regularized iterated Darboux transformation $S_\varepsilon$.
\begin{lemma} \label{lem:technical_estimates}
\begin{itemize}[leftmargin=0.63cm]
\item[(a)] For any $0 \leq \varepsilon \leq 1$ and $f \in L^2$, it holds
\begin{equation} \label{equ:lemma_technical_est_basic_Fourier}
 \begin{aligned}
  \|X_\varepsilon \px^m f\| \leq C \varepsilon^{-\frac{m}{2}} \|f\|, \quad 0 \leq m \leq 4. 
 \end{aligned}
\end{equation} 

\item[(b)] There exists $\varepsilon_1 > 0$ small such that for any $0 < \varepsilon \leq \varepsilon_1$, $K \geq 1$, and $f \in L^2$, we have 
\begin{equation} \label{equ:lemma_technical_est_sechXepspx}
 \begin{aligned}
  \Bigl\| \sech\Bigl(\frac{x}{K}\Bigr) X_\varepsilon \px^m f \Bigr\| &\leq C \varepsilon^{-\frac{m}{2}} \Bigl\| X_\varepsilon \Bigl[ \sech\Bigl(\frac{x}{K}\Bigr) f \Bigr] \Bigr\|, \quad 0 \leq m \leq 4,
 \end{aligned}
\end{equation}
and 
\begin{equation} \label{equ:lemma_technical_est_coshXepssech}
 \begin{aligned}
  \Bigl\| \cosh\Bigl(\frac{x}{K}\Bigr) X_\varepsilon \Bigl[ \sech\Bigl(\frac{x}{K}\Bigr) \Bigr] f \Bigr\| &\leq C \|X_\varepsilon f\|.
 \end{aligned}
\end{equation}
\end{itemize}
\end{lemma}
\begin{proof}
The bounds~\eqref{equ:lemma_technical_est_basic_Fourier} follow by elementary Fourier analysis. We establish the estimates~\eqref{equ:lemma_technical_est_sechXepspx} and~\eqref{equ:lemma_technical_est_coshXepssech} following closely the strategy of the proof of~\cite[Lemma 4.7]{KMMV20}. We begin with the estimate~\eqref{equ:lemma_technical_est_sechXepspx} in the case $m=0$. Let 
\begin{equation*}
 h(x) := \sech\Bigl(\frac{x}{K}\Bigr) X_\varepsilon f, \quad k(x) := X_\varepsilon \Bigl[ \sech\Bigl(\frac{x}{K}\Bigr) f \Bigr].
\end{equation*}
Then our goal is to show that $\|h\| \leq C \|k\|$ uniformly for all small $\varepsilon > 0$ and $K \geq 1$. We compute 
\begin{equation*}
 \begin{aligned}
  f &= (1-\varepsilon \px^2)^2 \Bigl[ \cosh\Bigl(\frac{x}{K}\Bigr) h \Bigr] = (1 - 2\varepsilon \px^2 + \varepsilon^2 \px^4) \Bigl[ \cosh\Bigl(\frac{x}{K}\Bigr) h \Bigr] \\
  &= \cosh\Bigl(\frac{x}{K}\Bigr) (1-\varepsilon\px^2)^2 h + \cosh\Bigl(\frac{x}{K}\Bigr) \biggl( - \frac{2\varepsilon}{K^2} h - \frac{4\varepsilon}{K} \tanh\Bigl(\frac{x}{K}\Bigr) h' \biggr) \\
  &\quad + \cosh\Bigl(\frac{x}{K}\Bigr) \biggl( \frac{\varepsilon^2}{K^4} h + \frac{4 \varepsilon^2}{K^3} \tanh\Bigl(\frac{x}{K}\Bigr) h' + \frac{6\varepsilon^2}{K^2} h'' + \frac{4\varepsilon^2}{K} \tanh\Bigl(\frac{x}{K}\Bigr) h''' \biggr).
 \end{aligned}
\end{equation*}
On the other hand, we have 
\begin{equation*}
 f = \cosh\Bigl(\frac{x}{K}\Bigr) (1-\varepsilon\px^2)^2 k,
\end{equation*}
and thus,
\begin{equation} \label{equ:lemma_technical_est_bound1}
 \begin{aligned}
  (1-\varepsilon\px^2)^2 k &= \biggl[ (1-\varepsilon\px^2)^2 - \frac{2\varepsilon}{K^2} + \frac{\varepsilon^2}{K^4} \biggr] h + \biggl( -\frac{4\varepsilon}{K} + \frac{4\varepsilon^2}{K^3} \biggr) \tanh\Bigl(\frac{x}{K}\Bigr) h' \\
  &\quad + \frac{6\varepsilon^2}{K^2} h'' + \frac{4\varepsilon^2}{K} \tanh\Bigl(\frac{x}{K}\Bigr) h'''.
 \end{aligned}
\end{equation}
Set 
\begin{equation*}
 \calT := \biggl[ (1-\varepsilon\px^2)^2 - \frac{2\varepsilon}{K^2} + \frac{\varepsilon^2}{K^4} \biggr].
\end{equation*}
Then \eqref{equ:lemma_technical_est_bound1} implies
\begin{equation} \label{equ:lemma_technical_est_bound2}
 \begin{aligned}
  h &= \calT^{-1} (1-\varepsilon\px^2)^2 k - \biggl( -\frac{4\varepsilon}{K} + \frac{4\varepsilon^2}{K^3} \biggr) \calT^{-1} \Bigl[ \tanh\Bigl(\frac{x}{K}\Bigr) h' \Bigr] \\
  &\quad - \frac{6\varepsilon^2}{K^2} \calT^{-1} \bigl[ h'' \bigr] - \frac{4\varepsilon^2}{K} \calT^{-1} \Bigl[ \tanh\Bigl(\frac{x}{K}\Bigr) h''' \Bigr].
 \end{aligned}
\end{equation}
Note that by elementary Fourier analysis there exists an absolute constant $C > 0$ such that uniformly for all small $\varepsilon > 0$ and all $K \geq 1$, we have the operator norm bounds
\begin{equation} \label{equ:lemma_technical_est_bound3}
 \begin{aligned}
  \bigl\| \calT^{-1} (1-\varepsilon\px^2)^2 \bigr\|_{L^2 \to L^2} &\leq C, \quad \bigl\| \calT^{-1} \px^m \bigr\|_{L^2 \to L^2} &\leq C \varepsilon^{-\frac{m}{2}}, \quad 0 \leq m \leq 4.
 \end{aligned}
\end{equation}
Upon rewriting 
\begin{equation*} 
 \begin{aligned}
  \tanh\Bigl(\frac{x}{K}\Bigr) h' &= \Bigl[ \tanh\Bigl(\frac{x}{K}\Bigr) h \Bigr]' - \frac{1}{K} \tanh'\Bigl(\frac{x}{K}\Bigr) h, \\
  \tanh\Bigl(\frac{x}{K}\Bigr) h''' &= \Bigl[ \tanh\Bigl(\frac{x}{K}\Bigr) h \Bigr]''' - \frac{3}{K} \Bigl[ \tanh' \Bigl(\frac{x}{K}\Bigr) h \Bigr]'' + \frac{3}{K^2} \Bigl[ \tanh''\Bigl(\frac{x}{K}\Bigr) h \Bigr]' - \frac{1}{K^3} \tanh'''\Bigl(\frac{x}{K}\Bigr) h,
 \end{aligned}
\end{equation*}
we conclude from \eqref{equ:lemma_technical_est_bound2} and \eqref{equ:lemma_technical_est_bound3} that
\begin{equation*}
 \|h\| \leq C \|k\| + C \varepsilon^\hf \|h\|.
\end{equation*}
The asserted estimate~\eqref{equ:lemma_technical_est_sechXepspx} in the case $m=0$ follows for sufficiently small $\varepsilon > 0$. The proofs of the estimates~\eqref{equ:lemma_technical_est_sechXepspx} in the cases $1 \leq m \leq 4$ are analogous, and the estimate~\eqref{equ:lemma_technical_est_coshXepssech} can be established similarly as in~\cite[Lemma 4.7]{KM22}. 
\end{proof}

In the following proposition we establish weighted $L^2$ bounds for the operator $S_\varepsilon$.
\begin{lemma} \label{lem:transfer1}
 For any $A > 0$ large, any $\varepsilon > 0 $ small, and any $u \in H^1$, 
 \begin{align}
  \|\sigma_A S_\varepsilon u\| &\lesssim \varepsilon^{-\thf} \|\sigma_A u\|, \label{equ:sigmaA_Seps_bound1} \\
  \|\sigma_A \px S_\varepsilon u\| &\lesssim \varepsilon^{-\thf} \|\sigma_A \px u\| + \|\rho u\|. \label{equ:sigmaA_Seps_bound2}
 \end{align}
\end{lemma}
\begin{proof}
We adapt the proof of \cite[Lemma 1]{KM22} to our setting.
By direct computation we find 
\begin{equation*}
 \calD_1 \calD_2 \calD_3 = \px^3 + \px^2 \cdot k_1 + \px \cdot k_2 + k_3
\end{equation*}
with 
\begin{equation*}
 \begin{aligned}
  k_1 &= -\frac{Z_1'}{Z_1}-\frac{Z_2'}{Z_2}-\frac{Z_3'}{Z_3}, \\
  k_2 &= 2\left(\frac{Z_1'}{Z_1}\right)'+ \left(\frac{Z_2'}{Z_2}\right)'+ \frac{Z_1'}{Z_1}\frac{Z_2'}{Z_2}+\frac{Z_1'}{Z_1}\frac{Z_3'}{Z_3}+\frac{Z_2'}{Z_2}\frac{Z_3'}{Z_3}, \\
  k_3 &= -\left(\frac{Z_1'}{Z_1}\right)''-\left(\frac{Z_1'}{Z_1}\right)'\frac{Z_3}{Z_3}-\left( \frac{Z_1'}{Z_1}\frac{Z_2'}{Z_2}\right)'-\frac{Z_1'Z_2'Z_3'}{Z_1Z_2Z_3}.  
 \end{aligned}
\end{equation*}
Note that $k_1$, $k_2$, and $k_3$ are smooth and bounded. Then the first estimate~\eqref{equ:sigmaA_Seps_bound1} follows from~\eqref{equ:lemma_technical_est_sechXepspx}.
Moreover, we obtain by direct computation
\begin{equation*}
 \begin{aligned}
  \px \calD_1 \calD_2 \calD_3 &= \px^4 + \px^2 \bigl( k_1 \px \bigr) + \px \bigl( (k_2 + k_1') \px \bigr) + (k_3 + k_2' + k_1'') \px + (k_3' + k_2'' + k_1''').
 \end{aligned}
\end{equation*}
Using that $k_1, k_2, k_3$ are bounded and that $k_1', k_2', k_3' \in \calY$, the second estimate~\eqref{equ:sigmaA_Seps_bound2} now also follows from~\eqref{equ:lemma_technical_est_sechXepspx}.
\end{proof}

Finally, we establish a key estimate that allows us to transfer weighted $L^2$ bounds for the transformed variable $S_\varepsilon u$ back to the original variable $u$, if the orthogonality condition $u = P_c u$ holds. We present an elementary proof that is inspired by computations in~\cite{LS1} and that is reminiscent of the proofs of \cite[Lemma 6]{KMM19} and of \cite[Lemma 2]{KM22}. See also \cite[Section 9]{CuccMaeda21}.

\begin{lemma} \label{lem:transfer2}
 Uniformly for all $0 < \varepsilon \leq 1$ and all functions $u \in H^1$ satisfying $u = P_c u$, it holds
 \begin{equation} \label{equ:key_transfer_back_estimate}
  \|\rho u\| \lesssim \|\rho S_\varepsilon u\| + \|\rho \px S_\varepsilon u\|.
 \end{equation}
\end{lemma}
\begin{proof}
Set
\begin{equation*}
 v := S_\varepsilon u = X_\varepsilon \calD_1 \calD_2 \calD_3 u.
\end{equation*}
Then we have 
\begin{equation*}
 \calD_1 \calD_2 \calD_3 u = (1-\varepsilon\px^2)^2 v = v - 2\varepsilon \px^2 v + \varepsilon^2 \px^4 v.
\end{equation*}
Since $u = P_c u$ by assumption, we obtain from~\eqref{equ:PcRDsg_equal_Pcg} that
\begin{equation*}
 u = P_c \calR[ \calD_1 \calD_2 \calD_3 u ] = P_c \calR[v] - 2\varepsilon P_c \calR[\px^2 v] + \varepsilon^2 P_c \calR[\px^4 v],
\end{equation*}
whence 
\begin{equation} \label{equ:lemma_key_transfer_rhou_bound_PcRs}
 \begin{aligned}
  \|\rho u\| \lesssim \|\rho P_c \calR[v]\| + \|\rho P_c \calR[\px^2 v]\| + \| \rho P_c \calR[\px^4 v]\|.
 \end{aligned}
\end{equation}
By repeated integration by parts, we obtain 
\begin{equation} \label{eqn:calRpx2v}
	\calR[\px^2 v] = \calR_3[v] + \calR_3\bigl[\calR_2\bigl[ (Z_3^{-1} Z_3') v \bigr] \bigr] + \tfrac{1}{4}\calR[v] - v(0)\calR_3[Z_2] - v'(0)\calR_3[\calR_2[Z_1]],
\end{equation}	
and
\begin{equation} \label{eqn:calRpx4v}
	\begin{split}
		\calR[\px^4 v] &= \px v + \tfrac{1}{4} \calR_3[v] + 2\calR_3\bigl[ (Z_3^{-1} Z_3') \px v \bigr] \\
		&\quad + \tfrac{1}{4} \calR_3\bigl[\calR_2\bigl[ (Z_3^{-1} Z_3') v \bigr] \bigr] - \calR_3\bigl[\calR_2\bigl[ (Z_2 (Z_2^{-1} Z_3^{-1} Z_3')') \px v \bigr] \bigr]  + \tfrac{1}{16} \calR[v] \\
		&\quad -v'(0)Z_3 - \bigl(v''(0) + \tfrac{1}{4}v(0)\bigr) \calR_3[Z_2] - \bigl(v'''(0) + \tfrac{1}{4}v'(0)\bigr)\calR_3[\calR_2[Z_1]].
	\end{split}
\end{equation}
We defer the presentation of the details of the derivation of the identities~\eqref{eqn:calRpx2v} and~\eqref{eqn:calRpx4v} to Appendix~\ref{sec:appendix}.
In view of~\eqref{equ:RZells_justYjs}, we have 
\begin{equation*}
 P_c Z_3 = 0, \quad P_c \calR_3[Z_2] = 0, \quad P_c \calR_3[\calR_2[Z_1]] = 0,
\end{equation*}
whence 
\begin{equation*}
 P_c \calR[\px^2 v] = P_c \calR_3[v] + P_c \calR_3\bigl[\calR_2\bigl[ (Z_3^{-1} Z_3') v \bigr] \bigr] + \tfrac{1}{4} P_c\calR[v],
\end{equation*}	
and
\begin{equation*}
\begin{split}
 P_c \calR[\px^4 v] &= P_c \px v + \tfrac{1}{4} P_c \calR_3[v] + 2 P_c \calR_3\bigl[ (Z_3^{-1} Z_3') \px v \bigr] \\
 &\quad + \tfrac{1}{4} P_c \calR_3\bigl[\calR_2\bigl[ (Z_3^{-1} Z_3') v \bigr] \bigr] - P_c \calR_3\bigl[\calR_2\bigl[ (Z_2 (Z_2^{-1} Z_3^{-1} Z_3')') \px v \bigr] \bigr]  + \tfrac{1}{16} P_c \calR[v].
\end{split}
\end{equation*}
Using that $Y_0, Y_1, Y_2 \in \calY$, $\|Z_3^{-1} Z_3'\|_{L^\infty} \lesssim 1$, and $\|Z_2 (Z_2^{-1} Z_3^{-1} Z_3')'\|_{L^\infty} \lesssim 1$, we obtain by the Cauchy-Schwarz inequality
\begin{equation} \label{equ:lemma_key_transfer_rho_PcR_v_bounds}
 \begin{aligned}
  \|\rho P_c \calR[v]\| &\lesssim \|\rho \calR_3 \calR_2 \calR_1 \rho^{-1}\|_{L^2 \to L^2} \|\rho v\|, \\
  \|\rho P_c \calR[\px^2 v]\| &\lesssim \bigl( \|\rho \calR_3 \rho^{-1}\|_{L^2 \to L^2} + \|\rho \calR_3 \calR_2 \rho^{-1}\|_{L^2 \to L^2} + \|\rho \calR_3 \calR_2 \calR_1 \rho^{-1}\|_{L^2 \to L^2} \bigr) \|\rho v\|, \\
  \|\rho P_c \calR[\px^4 v]\| &\lesssim \bigl( 1+ \|\rho \calR_3 \rho^{-1}\|_{L^2 \to L^2} + \|\rho \calR_3 \calR_2 \rho^{-1}\|_{L^2 \to L^2} + \|\rho \calR_3 \calR_2 \calR_1 \rho^{-1}\|_{L^2 \to L^2} \bigr) \\
  &\qquad \qquad \qquad \qquad \qquad \qquad \qquad \qquad \qquad \qquad \qquad \qquad \times \bigl( \|\rho v\| + \|\rho \px v\| \bigr).
 \end{aligned}
\end{equation}
The integral operators 
\begin{equation*}
 \bigl(\rho \calR_\ell \rho^{-1} f\bigr)(x) = \int K_\ell(x,y) f(y) \, \ud y, \quad 1 \leq \ell \leq 3,
\end{equation*}
have exponentially localized integral kernels 
\begin{equation*}
 K_\ell(x,y) = Z_\ell(x) Z_\ell(y)^{-1} \rho(x) \rho(y)^{-1} \bigl( \mathds{1}_{[0,\infty)}(x) \mathds{1}_{[0,x]}(y) - \mathds{1}_{(-\infty,0)}(x) \mathds{1}_{[x,0]}(y) \bigr).
\end{equation*}
By Schur's test, we conclude
\begin{equation*}
 \|\rho \calR_\ell \rho^{-1}\|_{L^2 \to L^2} \lesssim 1, \quad 1 \leq \ell \leq 3,
\end{equation*}
and thus,
\begin{equation*}
 \begin{aligned}
  \|\rho \calR_3 \rho^{-1}\|_{L^2 \to L^2} + \|\rho \calR_3 \calR_2 \rho^{-1}\|_{L^2 \to L^2} + \|\rho \calR_3 \calR_2 \calR_1 \rho^{-1}\|_{L^2 \to L^2} \lesssim \sum_{k=1}^3 \prod_{\ell=k}^3 \|\rho \calR_\ell \rho^{-1}\|_{L^2 \to L^2} \lesssim 1.
 \end{aligned}
\end{equation*}
The asserted estimate~\eqref{equ:key_transfer_back_estimate} now follows from \eqref{equ:lemma_key_transfer_rhou_bound_PcRs} and \eqref{equ:lemma_key_transfer_rho_PcR_v_bounds}.
\end{proof}

\section{Virial estimate for the transformed equation} \label{sec:virial_transformed_equation}

At this point we pass to the equation for the transformed variable $S_\varepsilon \bm{u}$, for which we carry out a second localized virial argument. In order to be able to close all estimates in the end, it is crucial that the obtained integrated-in-time localized estimates for the transformed variable $S_\varepsilon u_1$ are only in terms of localized bounds for the original variable $u_1$ that come with additional smallness.

\begin{proposition} \label{prop:virial_v}
 For any $\varepsilon > 0$ small, any $A > 0$ large, any $\delta > 0$ small (depending on $\varepsilon$ and $A$), and any $T > 0$, 
 \begin{equation} \label{equ:virial_v_bound}
  \begin{aligned}
  \int_0^T \Bigl( \|\rho S_\varepsilon u_1\|^2 + \|\rho \px S_\varepsilon u_1\|^2 \Bigr) \, \ud t &\lesssim A \delta^2 + \int_0^T |\bm{z}|^4 \, \ud t + \frac{1}{\sqrt{A}} \int_0^T \bigl( b_+^2 + b_-^2 \bigr) \, \ud t \\
  &\quad + \frac{1}{\sqrt{A}} \int_0^T \Bigl( \|\sigma_A \px u_1\|^2 + \frac{1}{A^2} \|\sigma_A u_1\|^2 + \|\rho u_1\|^2 \Bigr) \, \ud t.
  \end{aligned}
 \end{equation}
\end{proposition}
\begin{proof}
We introduce the transformed variables 
\begin{equation*}
 \begin{aligned}
  v_1 = S_\varepsilon u_1, \quad v_2 = S_\varepsilon u_2.
 \end{aligned}
\end{equation*}
Observe that $v_1$ and $v_2$ are odd.
From~\eqref{equ:pdeode_system_u} and the conjugation identity~\eqref{equ:conjugation_identity}, we obtain
\begin{equation} \label{equ:transformed_equation}
 \left\{ \begin{aligned}
  \pt v_1 &= v_2, \\
  \pt v_2 &= -(-\px^2 + 1) v_1 + S_\varepsilon N^\perp.
 \end{aligned} \right.
\end{equation}
We use the virial functional 
\begin{equation*}
 \begin{aligned}
  \calK := \int \biggl( \Psi_{A,B} \px v_1 + \frac12 \Psi_{A,B}' v_1 \biggr) v_2,
 \end{aligned}
\end{equation*}
where we recall that $\Psi_{A,B} = \chi_A^2 \Phi_B$.
Set 
\begin{equation*}
 \tilv_1 := \chi_A \zeta_B v_1.
\end{equation*}
Observe that $\tilv_1$ is also odd.
By the general virial computation~\eqref{equ:prelim_general_virial_identity} we have 
\begin{equation} \label{equ:pt_second_virial}
 \begin{aligned}
  \pt \calK &= \biggl( - \int \Psi_{A,B}' (\px v_1)^2 + \frac14 \int \Psi_{A,B}''' v_1^2 \biggr) + \int \biggl( \Psi_{A,B} \px v_1 + \frac12 \Psi_{A,B}' v_1 \biggr) S_\varepsilon N^\perp \\
  &= K_1 + K_2.
 \end{aligned}
\end{equation}
Proceeding as in~\cite[Sect. 4.3]{KMM19} and the proof of \cite[Proposition 3]{KM22}, we compute that
\begin{equation*}
 K_1 = - \int \Bigl( (\px \tilv_1)^2 + V_B \tilv_1^2 \Bigr) + \wtilK_1
\end{equation*}
with 
\begin{equation*}
 V_B = \frac12 \biggl( \frac{\zeta_B''}{\zeta_B} - \frac{(\zeta_B')^2}{\zeta_B^2} \biggr) = - \frac{1}{2B^2} \sech^2\Bigl(\frac{x}{B}\Bigr)
\end{equation*}
and
\begin{equation*}
 \begin{aligned}
  \wtilK_1 &= \frac14 \int (\chi_A^2)' (\zeta_B^2)' v_1^2 + \frac12 \int \bigl( 3 (\chi_A')^2 + \chi_A \chi_A'' \bigr) \zeta_B^2 v_1^2 \\
  &\quad - \int (\chi_A^2)' \Phi_B (\px v_1)^2 + \frac14 \int (\chi_A^2)''' \Phi_B v_1^2.
 \end{aligned}
\end{equation*}

We first obtain a coercive bound for the main quadratic form in the term $K_1$, up to controllable error terms.
\begin{lemma} \label{lem:prop_second_virial_lem1}
 There exists $\theta > 0$ such that
 \begin{equation} \label{equ:prop_second_virial_lower_bound}
  \int \Bigl( (\px \tilv_1)^2 + V_B \tilv_1^2 \Bigr) \geq \theta \bigl( \|\rho \px v_1\|^2 + \|\rho v_1\|^2 \bigr) - \frac{C}{A} \Bigl( \|\sigma_A \px u_1\|^2 + \frac{1}{A^2} \|\sigma_A u_1\|^2 \Bigr).
 \end{equation}
\end{lemma}
\begin{proof}
We recall from \cite[p. 926]{KMM17_short} that for every $\lambda > 0$ the Schr\"odinger operator 
\begin{equation*}
 -\px^2 - \frac{2}{\lambda^2} \sech^2\Bigl(\frac{x}{\lambda}\Bigr)
\end{equation*}
has only one negative discrete eigenvalue with a corresponding even eigenfunction.
Thus, for any $\lambda > 0$ and any odd function $f$, we have 
\begin{equation} \label{equ:prop_second_virial_sG_lower_bound}
 \int (\px f)^2 \geq \frac{2}{\lambda^2} \int \sech^2\Bigl(\frac{x}{\lambda}\Bigr) f^2.
\end{equation}
Since the variable $\tilv_1$ is odd, we can conclude as in~\cite[Lemma 4.1]{KMM17} and~\cite[Lemma 2.1]{KMM17_short} that
\begin{equation*}
 \begin{aligned}
  \int \Bigl( (\px \tilv_1)^2 + V_B \tilv_1^2 \Bigr) = \frac34 \int (\px \tilv_1)^2 + \frac14 \int \biggl( (\px \tilv_1)^2 - \frac{2}{B^2} \sech^2\Bigl(\frac{x}{B}\Bigr) \tilv_1^2 \biggr) \geq \frac34 \int (\px \tilv_1)^2.
 \end{aligned}
\end{equation*} 
Recalling that $\rho(x) = \sech^2(\frac{x}{20})$ and invoking~\eqref{equ:prop_second_virial_sG_lower_bound} once more with $\lambda = 20$, we obtain
\begin{equation} \label{equ:prop_second_virial_coercive_bound}
 \begin{aligned}
  \int \Bigl( (\px \tilv_1)^2 + V_B \tilv_1^2 \Bigr) \geq \frac12 \int (\px \tilv_1)^2 + \frac14 \frac{2}{20^2} \int \rho \tilv_1^2 \geq \frac{1}{800} \int \rho \, \Bigl( (\px \tilv_1)^2 + \tilv_1^2 \Bigr).
 \end{aligned}
\end{equation} 
Using Lemma~\ref{lem:transfer1}, it is straightforward to adapt the proof of the estimate (33) in \cite{KM22} to obtain for $A$ large (depending on $\varepsilon$) that
\begin{equation} \label{equ:prop_second_virial_bound_rhov_by_rhotilv}
 \begin{aligned}
  \|\rho \px v_1\|^2 + \|\rho v_1\|^2 \lesssim \|\rho^\hf \px \tilv_1\|^2 + \|\rho^\hf \tilv_1\|^2 + \frac{1}{A} \Bigl( \|\sigma_A \px u_1\|^2 + \frac{1}{A^2} \|\sigma_A u_1\|^2 \Bigr).
 \end{aligned}
\end{equation}
Now the asserted estimate~\eqref{equ:prop_second_virial_lower_bound} is a consequence of the estimates~\eqref{equ:prop_second_virial_coercive_bound} and \eqref{equ:prop_second_virial_bound_rhov_by_rhotilv}.
\end{proof}

Next, we estimate the term $\widetilde{K}_1$.
\begin{lemma} \label{lem:prop_second_virial_lem2}
For $A$ large (depending on $\varepsilon$), we have 
\begin{equation} \label{equ:prop_second_virial_wtilK1_bound}
 \begin{aligned}
  |\widetilde{K}_1| \lesssim \frac{1}{\sqrt{A}} \Bigl( \|\sigma_A \px u_1\|^2 + \frac{1}{A^2} \|\sigma_A u_1\|^2 + \|\rho u_1\|^2 \Bigr).
 \end{aligned}
\end{equation}
\end{lemma}
\begin{proof}
 Using Lemma~\ref{lem:transfer1} the claim follows for $A$ large (depending on $\varepsilon)$ by proceeding exactly as in the proof of~\cite[Lemma 4]{KM22}.
\end{proof}

Finally, we bound the term $K_2$.
\begin{lemma} \label{lem:prop_second_virial_lem3}
We have for some constant $C > 0$ that
\begin{equation}
 \begin{aligned}
  |K_2| &\leq \frac{\theta}{2} \bigl( \|\rho \px v_1\|^2 + \|\rho v_1\|^2 \bigr) + \frac{C}{\theta} |\bm{z}|^4 \\
  &\quad + \frac{C}{\sqrt{A}} \Bigl( \|\sigma_A \px u_1\|^2 + \frac{1}{A^2} \|\sigma_A u_1\|^2 + \|\rho u_1\|^2 + (b_+^2 + b_-^2) \Bigr).  
 \end{aligned}
\end{equation}
\end{lemma}
\begin{proof}
 We split the term $K_2$ into two parts 
 \begin{equation*}
 \begin{aligned}
  K_2 &= \int \biggl( \Psi_{A,B} \px v_1 + \frac12 \Psi_{A,B}' v_1 \biggr) S_\varepsilon \bigl( P_c(Y_2^2) z_1^2 \bigr) +  \int \biggl( \Psi_{A,B} \px v_1 + \frac12 \Psi_{A,B}' v_1 \biggr) S_\varepsilon \bigl( N^\perp - P_c(Y_2^2) z_1^2 \bigr) \\
  &= K_{2,1} + K_{2,2}.
 \end{aligned}
 \end{equation*}
 For the first term $K_{2,1}$ we use Lemma~\ref{lem:technical_estimates}, $|\Psi_{A,B}| + |\Psi_{A,B}'| \lesssim 1$, and $Y_2 \in \calY$, to conclude by the Cauchy-Schwarz inequality
 \begin{equation*}
 \begin{aligned}
  |K_{2,1}| &\lesssim |z_1|^2 \bigl( \|\rho \px v_1\| + \|\rho v_1\| \bigr) \bigl\| \rho^{-1} S_\varepsilon \bigl( P_c(Y_2^2) \bigr) \bigr\| \\
  &\lesssim |z_1|^2 \bigl( \|\rho \px v_1\| + \|\rho v_1\| \bigr) \bigl\| \rho^{-1} \calD_1 \calD_2 \calD_3 P_c(Y_2^2) \bigr\| \\
  &\leq \frac{C}{\theta} |\bm{z}|^4 + \frac{\theta}{2} \bigl( \|\rho \px v_1\|^2 + \|\rho v_1\|^2 \bigr).
 \end{aligned}
 \end{equation*}
 It remains to bound the second term $K_{2,2}$. Noting that $Y_0, Y_2 \in \calY$, we have the pointwise estimate 
 \begin{equation*}
 \begin{aligned}
  \bigl| N - Y_2^2 z_1^2 \bigr| \lesssim u_1^2 + (b_+^2 + b_-^2) \rho^6 + |z_1| |u_1| \rho^3 + |z_1| (|b_+| + |b_-|) \rho^6.
 \end{aligned}
 \end{equation*}
 Using~\eqref{equ:smallness} and $Y_0, Y_2 \in \calY$, it follows that
 \begin{equation*}
 \begin{aligned}
  \bigl\| \sigma_A \bigl( N^\perp - P_c(Y_2^2) z_1^2 \bigr) \bigr\| \lesssim \delta \bigl( \|\sigma_A u_1\| + |b_+| + |b_-| \bigr).
 \end{aligned}
 \end{equation*}
 Thus, using $\sigma_A \gtrsim 1$ on $[-2A, 2A]$ and Lemma~\ref{lem:transfer1}, we obtain
 \begin{equation*}
 \begin{aligned}
  |K_{2,2}| &\lesssim \bigl( \|\sigma_A \px v_1\| + \|\sigma_A v_1\| \bigr) \bigl\| \sigma_A S_\varepsilon \bigl( N^\perp - P_c(Y_2^2) z_1^2 \bigr) \bigr\| \\
  &\lesssim \varepsilon^{-3} \bigl( \|\sigma_A \px u_1\| + \|\sigma_A u_1\| + \|\rho u_1\| \bigr) \bigl\| \sigma_A \bigl( N^\perp - P_c(Y_2^2) z_1^2 \bigr) \bigr\| \\
  &\lesssim \varepsilon^{-3} \delta \bigl( \|\sigma_A \px u_1\| + \|\sigma_A u_1\| + \|\rho u_1\| \bigr) \bigl( \|\sigma_A u_1\| + |b_+| + |b_-| \bigr) \\
  &\lesssim \delta \varepsilon^{-3} A^2 \Bigl( \|\sigma_A \px u_1\|^2 + \frac{1}{A^2} \|\sigma_A u_1\|^2 + \|\rho u_1\|^2 + b_+^2 + b_-^2 \Bigr).
 \end{aligned}
 \end{equation*}
 Now the claim follows upon choosing $\delta$ sufficiently small depending on $A$ and $\varepsilon$.
\end{proof}

By Lemma~\ref{lem:transfer1} and \eqref{equ:smallness}, we have
\begin{equation*}
 |\calK| \lesssim \|\sigma_A \px v_1\|^2 + \|\sigma_A v_1\|^2 + \|\sigma_A v_2\|^2 \lesssim \varepsilon^{-3} \delta^2 \lesssim A \delta^2.
\end{equation*}
The estimate of Proposition~\ref{prop:virial_v} now follows from Lemmas~\ref{lem:prop_second_virial_lem1}, \ref{lem:prop_second_virial_lem2}, and \ref{lem:prop_second_virial_lem3} upon integrating~\eqref{equ:pt_second_virial} in time.
\end{proof}

\section{Conclusion of the proof of Theorem~\ref{thm:codim_asymptotic_stability}} \label{sec:conclusion} 

Finally, we combine Propositions~\ref{prop:virial_u}, \ref{prop:virial_z}, \ref{prop:virial_b}, \ref{prop:virial_v}, and Lemma~\ref{lem:transfer2} to conclude the proof of Theorem~\ref{thm:codim_asymptotic_stability} via a standard argument as in~\cite[Section 6]{KM22}.
From Proposition~\ref{prop:virial_b} we obtain by invoking Proposition~\ref{prop:virial_z} for any $T > 0$ that
\begin{equation}\label{eqn:virial b}
\begin{split}
 \int_0^T \bigl( b_+^2 + b_-^2 \bigr) \, \ud t &\lesssim A \delta^2 + \delta \int_0^T \|\rho u_1\|^2 \, \ud t + \frac{1}{\sqrt{A}} \int_0^T \bigl( b_+^2 + b_-^2 \bigr) \, \ud t \\
 &\quad + \frac{1}{\sqrt{A}} \int_0^T \Bigl( \Vert \sigma_A \partial_x u_1\Vert^2 + \frac{1}{A^2} \Vert \sigma_A u_1 \Vert^2 + \frac{1}{A^2} \Vert \sigma_Au_2 \Vert^2 \Bigr) \, \ud t.
\end{split}
\end{equation}
Combining Lemma~\ref{lem:transfer2} and Proposition~\ref{prop:virial_v}, and invoking Proposition~\ref{prop:virial_z}, we find for any $T > 0$, 
\begin{equation} \label{eqn:virial rho u}
\begin{aligned}
 \int_0^T \|\rho u_1\|^2 \, \ud t &\lesssim \int_0^T \bigl( \|\rho S_\varepsilon u_1\|^2 + \|\rho \px S_\varepsilon u_1\|^2 \bigr) \, \ud t \\
 &\lesssim A \delta^2 + \frac{1}{\sqrt{A}} \int_0^T \|\rho u_1\|^2 \, \ud t + \frac{1}{\sqrt{A}} \int_0^T \bigl( b_+^2 + b_-^2 \bigr) \, \ud t \\ 
 &\quad + \frac{1}{\sqrt{A}}\int_0^T \Bigl( \Vert \sigma_A \partial_x u_1\Vert^2 + \frac{1}{A^2} \Vert \sigma_A u_1 \Vert^2 + \frac{1}{A^2} \Vert \sigma_Au_2 \Vert^2 \Bigr) \, \ud t.
\end{aligned}
\end{equation}
Now combining the two preceding estimates \eqref{eqn:virial b} and \eqref{eqn:virial rho u} with Proposition~\ref{prop:virial_z} yields for any $T>0$ that
\begin{equation} \label{equ:concluding_theorem_combined_est1}
\begin{split}
 &\int_0^T \bigl( \Vert \rho u_1 \Vert^2 + \vert \bm{z} \vert^4 + b_+^2 + b_-^2 \bigr) \, \ud t \\
 &\lesssim A\delta^2 + \biggl( \delta + \frac{1}{\sqrt{A}} \biggr) \int_0^T \Vert \rho u_1 \Vert^2 \, \ud t + \frac{1}{\sqrt{A}} \int_0^T \bigl( b_+^2 + b_-^2 \bigr) \, \ud t \\
 &\quad + \frac{1}{\sqrt{A}} \int_0^T \Bigl(\Vert \sigma_A \partial_x u_1\Vert^2 + \frac{1}{A^2} \Vert \sigma_A u_1 \Vert^2 + \frac{1}{A^2} \Vert \sigma_Au_2 \Vert^2 \Bigr) \, \ud t.
\end{split}
\end{equation}
To conclude, we use Proposition \ref{prop:virial_u} to bound the last term on the right-hand side of the preceding estimate~\eqref{equ:concluding_theorem_combined_est1}. Then choosing $A$ sufficiently large and $\delta$ sufficiently small, we obtain for any $T > 0$, 
\begin{equation} \label{eqn:virial_rho}
 \int_0^T \left(\Vert \rho u_1 \Vert^2 + \vert \bm{z} \vert^4 + b_+^2 + b_-^2\right)\ud t \lesssim A\delta^2.
\end{equation}
We now fix such an $A$. Let
\begin{equation*}
 \calM := \Vert \sigma_A \partial_x u_1 \Vert^2 + \Vert \sigma_A u_1 \Vert^2 +\Vert \sigma_A u_2 \Vert^2 +  \vert \bm{z} \vert^4 + b_+^2 + b_-^2.
\end{equation*}
By Proposition \ref{prop:virial_u} and \eqref{eqn:virial_rho}, we obtain
\begin{equation} \label{equ:integral_M_to_infinity}
	\int_0^\infty \calM(t) \, \ud t \lesssim \delta^2.
\end{equation}
Hence, there exists a sequence of times $t_n \to \infty$ such that $\lim_{n \to \infty} \calM(t_n) = 0$. Inserting \eqref{equ:pdeode_system_u}, \eqref{equ:pt_modzsquared}, and integrating by parts, we compute
\begin{equation*}
\begin{split}
 \pt \calM &= 2\int \sigma_A^2 \Bigl( (\partial_x u_1)(\partial_x u_2) + u_1u_2 + u_2 \bigl( -Lu_1 + N^\perp \bigr) \Bigr) \\
 &\quad \quad + 4 \mu^{-1} z_2 \vert \bm{z} \vert^2 N_2 + 2\nu (b_+^2 - b_-^2) + \nu^{-1} (b_+ - b_-)N_0 \\
 &= 2 \int \Bigl( - 2 \sigma_A \sigma_A' (\partial_x u_1) u_2 + 2 \sigma_A^2 Q u_1 u_2 + \sigma_A^2 u_2 N^\perp \Bigr) \\
 &\quad \quad + 4 \mu^{-1} z_2 \vert \bm{z} \vert^2 N_2 + 2\nu (b_+^2 - b_-^2) + \nu^{-1} (b_+ - b_-)N_0.
\end{split}
\end{equation*}
Using $\vert \sigma_A' \vert \lesssim \sigma_A$, \eqref{equ:smallness}, and the estimates
\begin{equation*}
	\begin{split}
		\vert N_0 \vert + \vert N_2 \vert &\lesssim \Vert \sigma_A u_1 \Vert^2 + z_1^2 + b_+^2 + b_-^2,\\
		\vert N \vert &\lesssim u_1^2 + (b_+^2 +b_-^2) Y_0^2 + z_1^2Y_2^2,
	\end{split}
\end{equation*}
we obtain by Cauchy-Schwarz that
\begin{equation*}
	\bigl| \pt \calM \bigr| \lesssim \calM.
\end{equation*}
For any $t\geq 0$ and large $n \in \N$, integrating over the time interval $[t,t_n]$ yields
\begin{equation*}
	\calM(t) \lesssim \calM(t_n) + \int_t^{t_n} \calM(s) \, \ud s.
\end{equation*}
Since $\lim_{n \to \infty} \calM(t_n) = 0$, taking the limit $n \to \infty$ yields
\begin{equation*}
	\calM(t) \lesssim \int_t^\infty \calM(s) \, \ud s,
\end{equation*}
which by~\eqref{equ:integral_M_to_infinity} implies 
\begin{equation*}
	\lim_{t \to \infty} \calM(t) = 0.
\end{equation*}
For any bounded interval $I \subset \R$ we have $\Vert (u_1,u_2) \Vert_{H^1(I) \times L^2(I)}^2 \lesssim_I \calM$, and thus
\begin{equation*}
	\lim_{t \to \infty} \, \Bigl( \vert \bm{z}(t) \vert + \vert b_+(t) \vert+ \vert b_-(t)\vert + \Vert (u_1,u_2) \Vert_{H^1(I) \times L^2(I)} \Bigr) = 0.
\end{equation*}
This finishes the proof of Theorem~\ref{thm:codim_asymptotic_stability}.

\begin{appendix} 
\section{Derivation of \eqref{eqn:calRpx2v} and~\eqref{eqn:calRpx4v}} \label{sec:appendix}
 
We first record that integrating by parts gives
\begin{equation} \label{eqn:ibp formula}
 \calR_\ell[\px v] = v - v(0) Z_\ell + \calR_\ell\bigl[ (Z_\ell^{-1} Z_\ell') v \bigr], \quad 1 \leq \ell \leq 3.
\end{equation}
Now we begin with the derivation of the identity~\eqref{eqn:calRpx2v}. Using \eqref{eqn:ibp formula}, we have 
\begin{equation} \label{equ:Rpx2v_deriv1}
 \calR_1[\px^2 v] = \px v - v'(0) Z_1 + \calR_1\bigl[ (Z_1^{-1} Z_1') \px v \bigr].
\end{equation}
Integrating by parts and using that $Z_1'(0) = 0$, we rewrite the last term on the right-hand side as
\begin{equation} \label{equ:Rpx2v_deriv2}
 \calR_1\bigl[ (Z_1^{-1} Z_1') \px v \bigr] = (Z_1^{-1} Z_1') v - \calR_1\bigl[ Z_1 (Z_1^{-2}Z_1')' v\bigr].
\end{equation}
Combining \eqref{equ:Rpx2v_deriv1}, \eqref{equ:Rpx2v_deriv2} and observing that $Z_1 (Z_1^{-2}Z_1')' = -\frac{1}{4}$, we obtain
\begin{equation} \label{equ:Rpx2v_deriv3}
 \calR_1[ \px^2 v] = \px v  - v'(0) Z_1 + (Z_1^{-1} Z_1') v + \tfrac{1}{4} \calR_1[v].
\end{equation}
Applying $\calR_2$ to \eqref{equ:Rpx2v_deriv3}, and using \eqref{eqn:ibp formula} to rewrite $\calR_2[\px v]$, we obtain
\begin{equation} \label{eqn:R2R1v''}
 \calR_2[\calR_1[\px^2 v]] = v + \calR_2\bigl[ (Z_1^{-1}Z_1' + Z_2^{-1}Z_2') v\bigr] + \tfrac{1}{4} \calR_2[\calR_1[v]] - v(0) Z_2 - v'(0)\calR_2[Z_1].
\end{equation}
Applying $\calR_3$ to \eqref{eqn:R2R1v''} and observing that $Z_1^{-1}Z_1' + Z_2^{-1}Z_2' = Z_3^{-1}Z_3'$, we arrive at the identity \eqref{eqn:calRpx2v}.

\medskip 

Next, we derive the identity~\eqref{eqn:calRpx4v}. 
Invoking \eqref{eqn:calRpx2v} for $\calR[\px^4 v]$, we obtain
\begin{equation} \label{equ:Rpx4v_deriv1}
 \calR[\px^4 v] = \calR_3[\px^2 v] + \calR_3\bigl[ \calR_2\bigl[ (Z_3^{-1} Z_3') \px^2 v \bigr]\bigr] + \tfrac{1}{4} \calR[\px^2 v] - v''(0) \calR_3[Z_2] - v'''(0) \calR_3[\calR_2[Z_1]].
\end{equation}
Now we rewrite the first three terms on the right-hand side of~\eqref{equ:Rpx4v_deriv1}.
Using \eqref{eqn:ibp formula}, we have 
\begin{equation} \label{equ:Rpx4v_deriv2}
 \calR_3[\px^2 v] = \px v - v'(0) Z_3 + \calR_3\bigl[ (Z_3^{-1} Z_3') \px v\bigr].
\end{equation}
Integrating by parts, we find that
\begin{equation} \label{equ:Rpx4v_deriv3}
 \calR_3\bigl[ \calR_2\bigl[ (Z_3^{-1} Z_3') \px^2 v \bigr]\bigr] = \calR_3\bigl[ (Z_3^{-1} Z_3) \px v \bigr] - \calR_3\bigl[ \calR_2\bigl[ (Z_2 (Z_2^{-1} Z_3^{-1} Z_3')') \px v \bigr].
\end{equation}
Invoking \eqref{eqn:calRpx2v} again, we have 
\begin{equation} \label{equ:Rpx4v_deriv4}
 \tfrac{1}{4} \calR[\px^2 v] = \tfrac{1}{4} \calR_3[v] + \tfrac{1}{4} \calR_3\bigl[ \calR_2\bigl[ (Z_3^{-1}Z_3') v\bigr]\bigr] + \tfrac{1}{16}\calR[v] - \tfrac{1}{4} v(0) \calR_3[Z_2] - \tfrac{1}{4} v'(0) \calR_3[\calR_2[Z_1]].
\end{equation}
Combining \eqref{equ:Rpx4v_deriv1}--\eqref{equ:Rpx4v_deriv4}, we obtain the identity~\eqref{eqn:calRpx4v}, as desired.
 
\end{appendix}

\bibliographystyle{amsplain}
\bibliography{references}

\end{document}